\documentclass[11pt]{amsart}

\usepackage[]{hyperref}
\usepackage{enumitem}
\usepackage[margin=0.8 in]{geometry}
\usepackage{stmaryrd}
\usepackage{url}
\usepackage{abstract}
\usepackage[all]{xy}
\usepackage{amssymb}
\usepackage{amsthm}
\usepackage{amsfonts}
\usepackage{amsmath}
\usepackage{fancyhdr}

\usepackage{pdfpages}

\title[Explicit counting of ideals in number fields of arbitrary degree]{Explicit counting of ideals in number fields\\ of arbitrary degree}
\author{Anton Fehnker}
\date{\today}

\address{Department of Mathematical Sciences, University of Copenhagen,
 Universitetsparken 5, 
2100 Copenhagen, Denmark.}
 \email{af@math.ku.dk}

\numberwithin{equation}{section}

\theoremstyle{plain}
\newtheorem{theorem}{Theorem}[section]
\newtheorem{corollary}[theorem]{Corollary}
\newtheorem*{mainth}{Theorem \ref{T:main}}
\newtheorem{lemma}[theorem]{Lemma}
\newtheorem{proposition}[theorem]{Proposition}

\theoremstyle{definition}
\newtheorem{definition}[theorem]{Definition}
\newtheorem{example}[theorem]{Example}

\theoremstyle{remark}
\newtheorem{remark}[theorem]{Remark}

\DeclareMathAlphabet{\mathpzc}{OT1}{pzc}{m}{it}

\newcommand{\mkw}{\mathbb{M}}
\newcommand{\dmk}{\mathbb{M}^+}
\newcommand{\lsp}{\mathbb{L}}
\newcommand{\F}{\mathcal{F}}
\newcommand{\G}{\mathcal{G}}
\newcommand{\Nl}{\mathcal{N}_{\log}}
\newcommand{\E}{\mathcal{E}}
\newcommand{\R}{\mathbb{R}}
\newcommand{\C}{\mathbb{C}}
\newcommand{\Q}{\mathbb{Q}}
\newcommand{\vol}{\operatorname{Vol}}
\newcommand{\y}{\mathpzc{y}}
\newcommand{\NN}{\mathcal{N}}
\newcommand{\OO}{\mathcal{O}}
\newcommand{\T}{\mathbb{T}}
\newcommand{\Z}{\mathbb{Z}}

\begin{document}
	
	\maketitle
	\begin{abstract}
		We implement methods from the geometry of numbers to give explicit estimates for the number of integral ideals in a number field. We pay particular attention to minimising the effect of the degree $n$ of the number field on the error term and avoid terms on the order of $n^{n^2}$. We do this by studying  fundamental domains for the action of multiplying with units of infinite order in Minkowski space. With some lattice theory we show that one can make different choices for such a fundamental domain, which yield a smaller error, especially when the degree of the field extension is large. We also adapt Schmidt's partition trick to this generalised setting.
	\end{abstract}

	\section{Introduction}

    Let $K$ be a number field of degree $n$ and with ring of integers $\mathcal{O}_K$. We let $\sigma_1, \dots , \sigma_n : K \to \C$ be the embeddings of $K$, ordered, such that $\sigma_1,\dots,\sigma_{r_1}$ are the real embeddings $K \hookrightarrow \R$, and $\sigma_{r_1+1},\dots,\sigma_{n}$  are the complex embeddings $K  \hookrightarrow \C$ with $\sigma_{r_1+i}=\overline{\sigma_{r_1+r_2+i}}$ for $i=1,\dots,r_2$. We note that $r_1$ and $r_2$ satisfy $r_1+2r_2=n$, and we let $r=r_1+r_2-1$.
	We denote the absolute value of the discriminant of $\OO_K$ by $\Delta_K$, the regulator by $R_K$, the number of roots of unity in $K$ by $\omega_K$, the class group of $\OO_K$ by $\operatorname{Cl}(\OO_K)$, and the class number by $h_K$.

	The problem of estimating the number of ideals of $\OO_K$, whose norm is bounded by some parameter $t$, is closely connected to properties of the Dedekind zeta function $\zeta_K$.
    
    Good estimates for the ideal counting function,
    $$\sigma_K(t)=\# \{   \mathfrak{b} \subseteq \mathcal{O}_K \text{ integral ideal} \mid 0< N(\mathfrak{b})\leq t\}, $$
    have been used in many applications, such as  the prime ideal theorem \cite{shapiro} and more explicit versions of the Chebotarev density theorem \cite{Korneel}\footnote{An effective version of the Chebotarev density theorem was first proved by Lagarias and Odlyzko \cite{EffCheb} with all constants made explicit by Winckler \cite{winckler_theoreme_2013}.}.

    From the analytic class number formula, we know that
	$$
	\lim_{t\to \infty} \frac{\sigma_K(t)}{t} = \kappa_K,
	$$
    where $\kappa_K=\frac{2^{r_1+r_2} \pi^{r_2} h_K R_K}{\omega_K \sqrt{\Delta_K}}$ is the residue of $\zeta_K(s)$ at $s=1$.
    In this paper we will use techniques from the geometry of numbers to explicitly bound the error term of this estimation, i.e., the quantity $\vert \sigma_K(t) - t \kappa_K \vert$, in terms of $t$ and invariants of the field $K$.
    
    There are well established estimates due to Weber \cite{Weber}, 
    $$
    \vert \sigma_K(t) -  t \kappa_K \vert \ll t^{\frac{n-1}{n}},
    $$
    and Landau \cite{Landau},
    $$
    \vert \sigma_K(t) -  t \kappa_K \vert \ll t^{\frac{n-1}{n+1}},
    $$
    with implied constants depending on $K$.
    
    In this article, we will show an explicit version of Weber's estimate, building on work of Debaene \cite{Korneel}. 
	
	For a class $\mathcal{C}$ in the ideal class group $\operatorname{Cl} (\OO_K)$, we define
	$$\sigma_K(t,\mathcal{C})=\# \{ \mathfrak{b} \in \mathcal{C}  \mid 0<N(\mathfrak{b})\leq t\}. $$
	Our main statement, which we prove in Section \ref{S:main}, is the following.
	
	\begin{theorem}\label{T:main}
		For $t\geq \left(10 n^2 \sqrt{\Delta_K} \right)^n$, we have
        \begin{alignat}{3}
            \vert \sigma_K(t,\mathcal{C}) - t \, h_K^{-1}\kappa_K  \vert &\leq  t^{\frac{n-1}{n}}\,E_1(n)\, \omega_K^{-1} R_K , \tag{A1} \label{mA1} \\
            \vert \sigma_K(t,\mathcal{C}) - t \, h_K^{-1}\kappa_K  \vert &\leq  t^{\frac{n-1}{n}}\,  E_2(n) \,  \omega_K^{-1}   R_K^{1/n}  \log^+\left(\left( 15n\, 2^{\frac{r}{2}} \right)^n R_K \right)^{\frac{(n-1)^2}{n}}  ,  \tag{A2} \label{mA2}  \\
			\vert \sigma_K(t) - t \, \kappa_K  \vert 
            &\leq t^{\frac{n-1}{n}}\, 6nE_1(n)\, \omega_K^{-1}   R_K \, h_K^{1/n}   \log^+(h_K)^{\frac{(n-1)^2}{n}}, \tag{B1} \label{mB1}  \\
			\vert \sigma_K(t) - t \, \kappa_K  \vert 
            &\leq t^{\frac{n-1}{n}}\, E_2(n)  \, \omega_K^{-1}   (R_K h_K)^{1/n}  \log^+\left(\left( 15n\, 2^{\frac{r}{2}} \right)^n R_K h_K \right)^{\frac{(n-1)^2}{n}} . \tag{B2} \label{mB2}
        \end{alignat}
	Here, $\log^+(x) = \max(1,\log(x))$, and we can take 
    $$E_1(n)=\frac{45e}{2} n^{7/2} 2^{4.1  n} \quad \text{ and } \quad E_2(n)=2025e^2  n^{11/2} 2^{4n} \left(n-1  \right)^{\frac{n-1}{2}}.$$
	\end{theorem}

    There are also explicit versions of Landau's estimate due to Sunley \cite{Sunley} and Lee \cite{lee_number_2023}, using methods from complex analysis. 
    In Section \ref{S:discuss} we will see that even though Landau's estimate is better asymptotically, our method can yield better results for smaller values of $t$, such as $t= \left(10 n^2 \sqrt{\Delta_K} \right)^n$ (see, in particular, Example \ref{Ex:zeta}).

    We will assume throughout that $n\geq2$, since the case $n=1$ with $\sigma_\Q(t)=\lfloor t \rfloor$ is trivial.
    
	A lot of the methods, we use, are similar to those Debaene implemented to obtain the following result.
	
	\begin{theorem}[{\cite[Corollary~2]{Korneel}}]\label{T:Debaene}
		For $t\geq 1$, we have
		\begin{alignat}{2}
			\vert \sigma_K(t,\mathcal{C}) - t \, h_K^{-1}\kappa_K  \vert &\leq  t^{\frac{n-1}{n}}\, n^{10n^2}  R_K^{1/n} \left( 1 + \log R_K \right)^{\frac{(n-1)^2}{n}}, \tag{A} \\
			\vert \sigma_K(t) - t \, \kappa_K  \vert &\leq t^{\frac{n-1}{n}}\, n^{10n^2}  \left( R_K h_K \right)^{1/n} \left( 1 + \log (R_K h_K) \right)^{\frac{(n-1)^2}{n}}. \tag{B}
		\end{alignat}	
	\end{theorem}
	
	In terms of the regulator $R_K$ and the class number $h_K$, the result of Debaene is very similar. However, the term $n^{10n^2}$ makes the estimate less precise for fields of larger degree.
    In terms of the degree, our results \ref{mA1} and \ref{mB1} only grow exponential in $n$ with $2^{4.1n}$ being the dominant term. Of course, the error depends linearly on the regulator in these cases.

    Our inequalities \ref{mA2} and \ref{mB2} have the same dependence on $R_K$ and $h_K$ as Theorem \ref{T:Debaene}. In terms of the degree, these estimates grow roughly like $n^{2.5n}$. Here $n^{0.5n}$ comes from $E_2(n)$, and from the $\log$-term we get another factor on the order of $\log(2^{n^2})^n=n^{2n}\log(2)^n$.

	We now give an overview of the different methods used in this paper.
	Just as Debaene, we use a method generally known as geometry of numbers, i.e., we apply techniques from lattice theory to a problem in number theory.
	
	One key ingredient, needed to get an explicit counting result, is a general method for counting lattice points in a given subset of $\R^m$. Given a lattice $\Lambda \subseteq \R^m$ and a measurable subset $X \subseteq \R^m$, the goal is to estimate the quantity $\vert \Lambda \cap X \vert$. A natural guess for this quantity is $\frac{\vol(X)}{\vol(\Lambda)}$, and there are several possibilities for bounding the error of this estimation.
	
	Debaene uses a theorem by Widmer \cite{Widmer}, where the error term includes a factor $n^{\frac{3n^2}{2}}$, if no bound for the orthogonality defect of the lattice $\Lambda$ is known. Since this is exactly the kind of term we want to avoid in this paper, and at the same time we are not aware of any strong bounds for the orthogonality defect of ideal lattices, we use a different method for estimating $\vert \Lambda \cap X \vert$.
	This is Proposition \ref{P:LatLipCounting}. We are able to avoid a term on the order of $n^{n^2}$ by focusing on a different invariant of the lattice $\Lambda$. While Widmer's result focuses on the successive minima of the lattice, our method is based around the covering radius. The downside to this is that we need the covering radius to be ``small'' in comparison to the size of $X$ (more specifically, we have the condition $r(\Lambda) \leq \sqrt{n-1} L$ in Proposition \ref{P:LatLipCounting}). This is the reason why our main result, Theorem \ref{T:main}, only holds for larger values of $t$.
 	
 	After this general lattice theory, we continue in Section \ref{S:geoNum} by taking a closer look at our specific setting. The first important steps here is the realisation that we have some degrees of freedom in choosing the fundamental domain for the action of multiplying with algebraic units in Minkowski space (Corollary \ref{C:fundDoms}). In particular, the possibility to choose any fundamental domain $D$ of the unit lattice to define $\F(D,I,v)$ is instrumental in reducing the dependency on the degree $n$.
 	
 	The second important step is the generalisation of Schmidt's partition trick in Lemma \ref{L:Schmidt}. This gives us even more freedom in choosing a nice domain to analyse. We show that we can partition the domain $\F(D,I,v)$ into smaller parts, and that it actually suffices to only consider one of these parts, which we denote $\F ( D_{\delta}, \left( 0,1\right], \nu )$. In return, we  have to analyse the intersection of this domain with a family of lattices $\Psi_{y_s}A_t$. This culminates in Proposition \ref{P:SigLat}.

    To apply Proposition \ref{P:LatLipCounting} and get our main result, we need two ingredients. Firstly, we need to show that $\F ( D_{\delta}, \left( 0,1\right], \nu )$ is of a certain Lipschitz class. This is the content of Section \ref{S:Lipschitz}.
    Secondly, we need a good bound for the covering radius of the lattices $\Psi_{\beta_s} A_t$, which is what we work out in Section \ref{S:Minima}.
 	Finally, we can combine everything into our main statement in Section \ref{S:main}, followed by a short discussion in Section \ref{S:discuss}.

	\section*{Acknowledgements}
    I want to thank Preda Mih\u{a}ilescu for introducing me to this topic, while I was studying at the University of Göttingen. He gave me insight into an unpublished working paper \cite{Preda}, in which he ponders the idea of considering different domains in Minkowski space to reduce or avoid orthogonality defects in the estimation of the ideal counting function.
	I am very grateful to have gotten this insight and the opportunity to develop these ideas further.

    I also want to thank Ian Kiming, Morten Risager, and Fabien Pazuki from the University of Copenhagen for supporting me in the later developments of this project, patiently listening to my ideas and giving constructive feedback.

	\section{Lattices}

    We start by discussing some general theory of lattices in a real vector space $\R^m$. We are especially interested in studying the intersection of a lattice $\Lambda\subseteq \R^m$ and other subsets $X\subseteq\R^m$. We use $\partial X$ to denote the boundary of $X$ and $\overline{X}$ for the topological closure $X\cup \partial X$. We write $B(r)$ to denote the closed ball of radius $r$ in $\R^m$, i.e.,
    $
    B(r)=\{x\in \R^m \mid \|x\|\leq r\}.
    $
    
	\begin{definition}
		Let $v_1,\dots, v_k \in \R^m$ be linearly independent vectors. Then 
		$$
		\Lambda = \mathbb{Z} v_1 + \dots + \mathbb{Z} v_k
		$$
		is a lattice of rank $k$ with basis $v_1,\dots v_k$. A lattice is called full, if $k=m$.
	\end{definition}
	
	\begin{definition}
		Let $\Lambda \subseteq \R^m$ be a lattice and let $V= \operatorname{Span}_\R \Lambda$.
		A fundamental domain of $\Lambda$ is a measurable subset $D\subseteq V$, which is a system of representatives for $V/\Lambda$, i.e., for every $v\in V$, there is a unique decomposition of the form $v=l+d$ with $l\in \Lambda$ and $d\in D$.
	\end{definition}
	
	\begin{example}
		Let $\Lambda \subseteq \R^m$ be a lattice with basis $v_1,\dots v_k$. We now describe some important fundamental domains.
		\begin{enumerate}
			\item The fundamental parallelepiped $$P=\left\{ \sum_{i=1}^k \mu_i v_i \biggm\vert \mu_i \in \left[0,1\right) \right\}$$ is a fundamental domain of $\Lambda$.
			\item Let $v_1^\ast, \dots, v_k^\ast$ be the Gram-Schmidt orthogonalised basis resulting from $v_1,\dots v_k$, i.e.,
			$$
			v_i^\ast = v_i - \sum_{j=1}^{i-1} \frac{\langle v_i, v_j^\ast \rangle}{\langle v_j^\ast, v_j^\ast \rangle} v_j^\ast.
			$$
			Then $$\left\{\sum_{i=1}^k \mu_i v_i^\ast \biggm\vert \mu_i \in \left[0,1\right)\right\}$$ is a fundamental domain of $\Lambda$.
			\item For any $v\in V=\operatorname{Span}_\R \Lambda$ and any fundamental domain  $D$ of $\Lambda$, the translated domain $v+D$ is also a fundamental domain of $\Lambda$. In particular, this means that the previous two fundamental domains can be centralised by taking $\mu_i \in \left[-\frac{1}{2},\frac{1}{2}\right)$.
		\end{enumerate}
	\end{example}
	
	All fundamental domains $D \subseteq V$ of a lattice $\Lambda$ have the same volume, the measure of the compact quotient space $V/\Lambda$. We denote this quantity by $\vol(\Lambda)$.

	\begin{definition}[Successive minima]
		Let $\Lambda$ be a lattice. The numbers
		$$
		\lambda_k (\Lambda):= \inf \{r \mid \text{there are } k \text{ linearly independent lattice points } l_1,\dots,l_k \in \Lambda \cap B(r) \}
		$$
		are called the successive minima of $\Lambda$.
	\end{definition}
	
	\begin{theorem}[Minkowski's second theorem {\cite{ CasselsGeoNum,Minkowski}}]\label{T:Mink}
		Let $\Lambda \subseteq \R^m$ be a full lattice with successive minima $\lambda_1,\dots,\lambda_m$. Then
		$$
		\frac{2^m}{m!} \vol(\Lambda) \leq \lambda_1 \dots \lambda_m \vol(B(1)) \leq 2^m \vol(\Lambda).
		$$ 
	\end{theorem}

	\begin{definition}
		Let $\Lambda \subseteq \R^m$ be a full lattice.
		The covering radius of $\Lambda$ is defined as
		$$
		r(\Lambda) := \inf\{ r \mid \R^m = \Lambda + B(r) \}.
		$$
	\end{definition}
	
	\begin{lemma}\label{L:RadMin}
		Let $\Lambda \subseteq \R^m$ be a full lattice. Then
		$$
		r(\Lambda) \leq \frac{ 2^{m-1} \sqrt{m} }{ \vol(B(1))} \frac{\vol(\Lambda)}{\lambda_1(\Lambda)^{m-1}}.
		$$
	\end{lemma}
	
	\begin{proof}
		Let $\lambda_1 , \dots , \lambda_m$ be the successive minima of $\Lambda$. There are linearly independent lattice points $a_1,\dots, a_m$ with $\|a_i\|\leq \lambda_m$. We consider the corresponding Gram-Schmidt orthogonalised basis $a_1^\ast,\dots, a_m^\ast$ with $\|a_i^\ast\| \leq \|a_i\|\leq \lambda_m$. The domain
		$$P=\left\{ \sum \mu_i a_i^\ast \mid \mu_i \in \left[ -\frac{1}{2}, \frac{1}{2} \right)\right\}$$
		satisfies $\R^m= \Lambda + P$, and, for $x\in P$, we have $\|x\|^2 \leq \sum \|\frac{a_i^\ast}{2}\|^2 \leq \frac{m  \lambda_m^2}{4} $, so $P\subseteq B\left(\frac{\sqrt{m}  \lambda_m}{2}\right)$. Hence, the covering radius satisfies $r(\Lambda) \leq \frac{\sqrt{m}  \lambda_m}{2} $.
		
		Finally, by Theorem \ref{T:Mink}, we have
		$$ \lambda_m \leq \frac{2^m \vol(\Lambda)}{\lambda_1 \dots \lambda_{m-1} \vol(B(1))}\leq \frac{2^m  \vol(\Lambda)}{\lambda_1^{m-1} \vol(B(1))}. $$
	\end{proof}

	\begin{remark}\label{R:Voronoi}
		The closed Voronoi cell of a full lattice $\Lambda\subseteq \R^m$ is
		$$
		\overline{C}=\left\{ x \in \R^m \mid  \forall l\in \Lambda  :\,\|x\| \leq \|x-l\|\right\}.
		$$
		By choosing a basis $v_1,\dots,v_m$ of $\Lambda$ and imposing an order $v_1 < v_2<\dots < v_m$, we give $\Lambda$ the structure of a totally ordered group. Using this we define
        $$
        C=\{x\in \overline{C}  \mid  \forall l\in \Lambda : \|x\|=\|x-l\| \Rightarrow l\geq 0\}.
        $$
        One easily checks that this is a fundamental domain of $\Lambda$.
        A decomposition of an element $v\in \R^m$ as $v=c+l$ with $c\in C$ and $l\in \Lambda$ is achieved by taking $l$ to be the closest lattice point to $v$ with respect to the Euclidean norm. For points $v$ that have an equal shortest distance to several $l\in \Lambda$, we must choose the smallest such $l$ with respect to the ordering on $\Lambda$.

        As a consequence, we have $\vol(C)=\vol(\Lambda)$.
		It is also easy to see that
		$$
		\sup\{ \|x\| \mid x \in C \} = r(\Lambda).
		$$
	\end{remark}

	\begin{lemma}\label{L:LatMantCounting}
		Let $\Lambda\subseteq \R^m$ be a full lattice with covering radius $r(\Lambda)$ and let $X\subseteq \R^m$ be a measurable, bounded set. Let $I_X=\{x \in X \mid d(x, \partial X) > r(\Lambda) \}$, $O_X=\{x \in  \R^m \mid d(x, X) \leq r(\Lambda) \}$, and $M_X=\{x \in \R^m \mid d(x, \partial X) \leq r(\Lambda) \} = O_X \setminus I_X $.
		Then we have
        \begin{alignat}{2}
            \frac{\vol(I_X)}{\vol (\Lambda)} \leq \vert X \cap \Lambda \vert &\leq \frac{\vol(O_X)}{\vol( \Lambda)}, \tag{1} \label{LatMantP1}\\
            \bigg\vert\vert X \cap \Lambda \vert - \frac{\vol(X)}{\vol(\Lambda)} \bigg\vert &\leq \frac{\vol(M_X)}{\vol (\Lambda)}. \tag{2}
        \end{alignat}
	\end{lemma}
	
	\begin{proof}
		Let $C$ be a fundamental domain of $\Lambda$, such that $\overline{C}$ is the closed Voronoi cell of $\Lambda$, as in Remark \ref{R:Voronoi}. We define $Y=C +(\Lambda\cap X)$, so $\vol (Y) = \vol(\Lambda) \cdot \vert \Lambda \cap X \vert$. For the first part, all we have to show is $I_X \subseteq Y \subseteq O_X$.
		
		Let $x\in I_X$. Then there is an $l \in \Lambda$ with $x \in l + C$. Since $\|x-l\| \leq r(\Lambda)$ and $d(x,\partial X)>r(\Lambda)$, we have $l\in X$, and therefore, $x\in Y$.
		
		Now let $y\in Y$. Then there is an $l\in \Lambda\cap X$ with $y\in l+C$, so $\|y-l\|\leq r(\Lambda)$, and therefore, $y\in O_X$. This proves \eqref{LatMantP1}.
		
		For the second part, we note that $X\setminus I_X \subseteq M_X$ and $O_X \setminus X \subseteq M_X$. From the first part, we also have 
		$$
		\frac{\vol(I_X)}{\vol \Lambda} - \frac{\vol(X)}{\vol \Lambda}\leq \vert X \cap \Lambda \vert - \frac{\vol(X)}{\vol \Lambda} \leq \frac{\vol(O_X)}{\vol \Lambda} - \frac{\vol(X)}{\vol \Lambda},
		$$ 
		so 
		$$\left\vert \vert X \cap \Lambda \vert - \frac{\vol(X)}{\vol \Lambda}  \right\vert \leq \frac{\max (\vol(X\setminus I_X), \vol(O_X\setminus X) )}{\vol \Lambda} \leq \frac{\vol(M_X)}{\vol \Lambda}. $$
	\end{proof}

	\begin{definition}
		We say that a set $S \subseteq \R^m$ is of Lipschitz class $\operatorname{Lip}(N,M,L)$ with $N,M\in \mathbb{N}$ and $L\in \R_{>0}$, if there are $M$ maps $\phi_1,\dots,\phi_M : [0,1]^N \to \R^m$, such that the images of $\phi_1,\dots,\phi_M$ cover all of $S$, and the maps satisfy the Lipschitz condition
		$$
		\| \phi_i (x) - \phi_i(y) \| \leq L  \|x-y\| \quad \forall x,y \in [0,1]^N,
		$$
		for $i=1,\dots,M$.
	\end{definition}

	\begin{proposition}\label{P:LatLipCounting}
		Let $X\subseteq \R^m$ be a bounded set and $\Lambda \subseteq \R^m$ a full lattice, such that $\partial X$ is of Lipschitz class $\operatorname{Lip}(m-1,M,L)$, and the covering radius of $\Lambda$ satisfies $r(\Lambda) \leq \sqrt{m-1} L$.
		Then $X$ is measurable and
		$$
		\bigg\vert\vert X \cap \Lambda \vert - \frac{\vol(X)}{\vol(\Lambda)} \bigg\vert \leq M   \left(\sqrt{m-1} L \right)^{m-1} 2^m \vol \left(B (1) \right) \frac{r(\Lambda)}{\vol(\Lambda)}.
		$$
	\end{proposition}
	\begin{proof}    
		Consider one of the maps $\phi_i : [0,1]^{m-1} \to \partial X$. We partition the interval $[0,1]$ into $Q$ equal parts for some $Q\in \mathbb{N}$. In this way, $[0,1]^{m-1}$ gets partitioned into $Q^{m-1}$ cubes of sidelength $1/Q$. Let $C$ be one of these cubes and let $c_0$ be the centre point of $C$. 
		Then for all $c\in C$ we have $\|c-c_0\|\leq \frac{\sqrt{m-1}}{2Q}$. Because of the Lipschitz-condition on $\phi_i$, the set $\{x\in \R^m \mid d(x,\phi_i(C)) \leq r(\Lambda) \} $ is contained in a ball centred at $\phi_i (c_0)$ with radius $L \frac{\sqrt{m-1}}{2Q} + r(\Lambda)$.
		This makes it possible to bound the volume of $M_X= \left\{ x \in \R^m \mid d(x, \partial X) \leq r(\Lambda) \right\} $ used in Lemma \ref{L:LatMantCounting} by
		$$
		\vol(M_X) \leq M  Q^{m-1}  \vol \left(B \left(L \frac{\sqrt{m-1}}{2Q} + r(\Lambda)\right) \right).
		$$
		
		The condition $r(\Lambda) \leq \sqrt{m-1} L$ means we can choose the integer $Q$, such that
		$$
		\frac{ \sqrt{m-1} L}{2 r(\Lambda)} \leq Q \leq  \frac{\sqrt{m-1} L}{ r(\Lambda)},
		$$
		and therefore, $L \frac{\sqrt{m-1}}{2Q} \leq r(\Lambda)$.
		
		Putting everything together, we get
		\begin{alignat*}{2}
			\vol(M_X) &\leq M  Q^{m-1}  \vol \left(B \left(L \frac{\sqrt{m-1}}{2Q} + r(\Lambda)\right) \right) \\
			&\leq M  \left(\frac{\sqrt{m-1} L}{ r(\Lambda)} \right)^{m-1}  \vol \left(B \left(2 r(\Lambda)\right) \right) \\
			&= M   \left(\frac{\sqrt{m-1} L}{ r(\Lambda)} \right)^{m-1} \left(2 r(\Lambda)\right)^m  \vol \left(B (1) \right) \\
			& = M   \left(\sqrt{m-1} L \right)^{m-1} 2^m \vol \left(B (1) \right) r(\Lambda).
		\end{alignat*}
		The statement now follows from Lemma \ref{L:LatMantCounting}.
	\end{proof}
	
    The last lattice-theoretic result, we will need, concerns the relation between a basis of a lattice and the corresponding Gram--Schmidt orthogonalised basis.
    The following is a consequence of the termination of the LLL-algorithm.
    
	\begin{lemma}[{\cite[Proposition~1.6]{lenstra_factoring_1982}}]\label{L:LLL}
		Every lattice $\Lambda$ has a basis $v_1,\dots, v_k$, such that $\|v_i\| \leq 2^{\frac{i-1}{2}} \|v_i^\ast \|$ for all $i=1,\dots,k$, where the $v_i^\ast$ form the corresponding Gram-Schmidt orthogonalised basis.
	\end{lemma}

	\section{Geometry of numbers}\label{S:geoNum}
    In this section, we recall how to reinterpret the problem of estimating $\sigma_K(t,\mathcal{C})=\# \{ \mathfrak{b} \in \mathcal{C}  \mid 0< N(\mathfrak{b})\leq t\}$ as a problem in lattice theory. To get started, we choose an integral ideal from the inverse class $\mathfrak{a}\in \mathcal{C}^{-1}$. Given an ideal $\mathfrak{b} \in \mathcal{C}$, satisfying $N(\mathfrak{b})<t$, the product $\mathfrak{ab}$ is a principal ideal $\alpha\mathcal{O}_K$ with $0<N(\alpha\mathcal{O}_K)\leq N(\mathfrak{a}) t$. Moreover, we have $\mathfrak{a} \mid \alpha\mathcal{O}_K$, or, equivalently, $\alpha \in \mathfrak{a}$. Because of the unique ideal factorisation in $\OO_K$, this gives a one-to-one correspondence, i.e.,
    $$
    \sigma_K(t,\mathcal{C})=\#\{\alpha\mathcal{O}_K \mid \alpha\in \mathfrak{a},\, 0<\vert N_{K/\Q}(\alpha) \vert\leq N(\mathfrak{a}) t \}.
    $$
    Two nonzero elements $\alpha,\alpha^\prime\in \mathfrak{a}$ generate the same principal ideal, if and only if  $\alpha/\alpha^\prime\in \mathcal{O}_K^\times$. Therefore, if we consider the set $E(\mathfrak{a},t)=\{\alpha \in \mathfrak{a} \mid 0< \vert N_{K/\Q}(\alpha)\vert \leq N(\mathfrak{a}) t \}$ with $\mathcal{O}_K^\times$ acting on it by multiplication, then $\sigma_K(t,\mathcal{C})$ is exactly the number of orbits under this action. To translate this into a lattice problem, we have to recall some classical theory from the geometry of numbers.
    
	\subsection{Important maps and spaces}\label{Ss:GeoNum}
		Recall the Minkowski embedding of a number field $K$ into the Minkowski space $\mkw:=\R^{r_1}\times \C^{r_2}$
		\begin{alignat*}{2}
			\phi : K &\rightarrow \mkw, \\
			\alpha &\mapsto (\sigma_i (\alpha))_{i=1}^{r_1+r_2}.
		\end{alignat*}
        We view $\mkw$ as an $n$-dimensional $\R$-algebra with a norm $\|\cdot \|$ given by $\| (x_i)_{i=1}^{r_1+r_2}\|=\left( \sum_{i=1}^{r_1+r_2} \vert x_i\vert^2  \right)^{1/2}$, where $\vert \cdot \vert$ is the absolute value on $\R$ or $\C$ respectively.
        
        It is well known (see, e.g., \cite[§5~Theorem~36]{marcus_number_2018}) that for any integral ideal $\mathfrak{a}\subseteq\mathcal{O}_K$, the image $\phi(\mathfrak{a})\subseteq \mkw$ is a full lattice with
        \begin{equation} 
            \vol(\phi(\mathfrak{a})) = 2^{-r_2} \sqrt{\Delta_K} N(\mathfrak{a}).
        \end{equation}

        We define the logarithmic space $\lsp:=\R^{r_1+r_2}$. We will study the map
		\begin{alignat*}{2}
			\ell : \mkw^\times &\rightarrow \lsp, \\
			(x_i) &\mapsto (n_i \log(\vert x_i \vert)),
		\end{alignat*}
        where $\mkw^\times =(\R^\times)^{r_1} \times (\C^\times)^{r_2}$ are the units of $\mkw$, and we let $n_i=1$ for $1\leq i\leq r_1$ and $n_i=2$ for $r_1<i\leq r_1+r_2$.
        
		We define the norm map
		\begin{alignat*}{2}
			\NN: \mkw &\rightarrow \R,\\
			(x_i)_{i=1}^{r_1+r_2} &\mapsto \prod_{i=1}^{r_1+r_2} \vert x_i \vert^{n_i}.
		\end{alignat*}
		Notice that for $\alpha \in K$, we have $\vert N_{K/\Q}(\alpha)\vert = \NN(\phi(\alpha))$.
		  There is a corresponding  linear map
		\begin{alignat*}{2}
			\Nl : \lsp &\rightarrow \R, \\
			(x_i) &\mapsto \sum_{i=1}^{r_1+r_2} x_i,
		\end{alignat*}
		which satisfies $\Nl\circ\ell = \log \circ \NN$. We define the hyperplane $H := \{ x \mid \Nl(x) = 0  \} \subseteq \lsp$. It is well-known that
		$$ \Gamma:=\ell \circ \phi(\mathcal{O}_K^\times) \subseteq H$$
		is a lattice of rank $r=r_1+r_2-1$, which we call the unit-lattice. We can see that $\vol(\Gamma)=\sqrt{r_1+r_2} R_K$.

        The Minkowski embedding $\phi$ is a ring homomorphism, and moreover, multiplication by a fixed element $y\in \mkw$ is $\R$-linear.
        We denote this multiplication-map by $\Psi_y$, i.e.,
		\begin{alignat*}{2}
			\Psi_y : \; \mkw &\to \mkw, \\
			 (x_i)_{i=1}^{r_1+r_2} &\mapsto \left( x_i y_i \right)_{i=1}^{r_1+r_2}.
		\end{alignat*}
        Observe that $\phi(\alpha \beta)=\Psi_{\phi(\alpha)}(\phi(\beta))$, and $\vert \det \Psi_y \vert = \NN(y)$. 
        
        Furthermore, we note that $\ell:\mkw^\times \to \lsp$ is a group homomorphism. On logarithmic space we define, for any element $y\in \lsp$, the translation map
        \begin{alignat*}{2}
			\tau_{y} : \;\lsp &\to \lsp, \\
			 x & \mapsto  x + y.
		\end{alignat*}
    The following lemma describes a key interaction between Minkowski space and logarithmic space.
	\begin{lemma}\label{L:TransDia}
		For $y \in \mkw^\times$, the diagram
	\[
	\xymatrix{
		\mkw^\times \ar[d]_-{\ell} \ar[r]^-{\Psi_y} & \mkw^\times  \ar[d]^-{\ell} \\
		\lsp  \ar[r]_-{\tau_{\ell(y)}} & \lsp }
	\]
	commutes.
	\end{lemma}
    \begin{proof}
        For $x,y\in \mkw^\times$, we have
        $$
        \ell \circ \Psi_y(x)=\ell(x\cdot y)=\ell(x) + \ell(y) =\tau_{\ell(y)} \circ \ell (x).
        $$
    \end{proof}
	
\subsection{Fundamental domains for \texorpdfstring{$\mathcal{O}_K^\times$}{OK*}-action}
    With the setup so far, we can describe the image of the set $E(\mathfrak{a},t)=\{\alpha \in \mathfrak{a} \mid 0< \vert N_{K/\Q}(\alpha)\vert \leq N(\mathfrak{a}) t \}$ under $\phi$
    as those lattice points $x\in \phi(\mathfrak{a)}$ that satisfy $0<\NN(x)\leq N(\mathfrak{a}) t $. We now describe, how we can count the number of orbits of the $\OO_K^\times$-action on the set $E(\mathfrak{a},t)$.
	
	\begin{definition}\label{D:fundDom}
		Let $D\subseteq H$, $I \subseteq \R^+$, and $v\in \lsp$, such that $\Nl(v) = 1$. We define
		$$\mathcal{E}(D,I,v):=\{h+\mu  v \mid h \in D, \mu \in \R, \exp(\mu)\in I \}$$
		and
		\begin{equation}
			\F(D,I,v):= \ell^{-1} ( \mathcal{E}(D,I,v) ).
		\end{equation}
	\end{definition}
	
	\begin{lemma}\label{L:fundsProperties}
		Let $D\subseteq H$ be a fundamental domain of the unit lattice $\Gamma$, and let $v\in \lsp$, such that $\Nl (v) = 1$. Let $I\subseteq \R^+$.
        \begin{enumerate}
            \item \label{i:normFund} We have
            $$
            \F(H,I,v)=\{ x \in \mkw \mid \NN(x) \in I \}.
            $$
            \item \label{i:unitFund} Fix a decomposition $\mathcal{O}_K^\times = W \oplus U$, where $W$ is the subgroup of roots of unity and $U\cong \mathbb{Z}^r$. Then, for every $x \in \F(H,I,v)$, there is a unique $\varepsilon \in U$, such that
            $$
            \Psi_{\phi(\varepsilon)}x \in \F(D,I,v),
            $$
            while, for any $\zeta \in W$,
            $$
            \Psi_{\phi(\zeta)} \F(D,I,v) = \F(D,I,v).
            $$
        \end{enumerate}
	\end{lemma}
    \begin{proof}
        For part (\ref{i:normFund}), observe that any element $x\in\lsp$ has a unique decomposition $x=h+\mu  v$ with $h\in H$ and $\mu \in \R$, and that $\Nl(x)=\Nl(h)+\mu \Nl(v)=\mu$. Thus, for $x\in \mkw^\times$,
        $$
        x \in \F(H,I,v) \Leftrightarrow \exp(\Nl(\ell(x)))\in I \Leftrightarrow \NN(x)=\exp (\log(\NN(x)))\in I.
        $$
        For part (\ref{i:unitFund}), observe that $\ell\circ\phi$ restricts to an isomorphism between $U$ and $\Gamma$. Given $x\in \F(D,I,v)$ we can decompose $\ell(x)=h+\mu v$. Now there exists a unique $l\in \Gamma$, such that $\tau_l (h)=h+ l \in D$. Letting $\varepsilon \in U$ be the preimage of $l$ under $\ell\circ\phi$, we see that $\Psi_{\phi(\varepsilon)}x \in \F(D,I,v)$, by Lemma \ref{L:TransDia}.
        
        Finally, $\Psi_{\phi(\zeta)} \F(D,I,v) = \F(D,I,v)$ follows from $\ell\circ\phi(W)=0$.
    \end{proof}

	We let $\mathfrak{a} \in \mathcal{C}^{-1}$ be an integral ideal and define, for $t>0$,
	\begin{equation}\label{E:A_t}
		A_t := \frac{1}{t N(\mathfrak{a})^{1/n}} \phi(\mathfrak{a}).
	\end{equation}
	We see that
	\begin{equation}
		\vol(A_t) = \frac{\sqrt{\Delta_K}}{2^{r_2}  t^n}.
	\end{equation}
    
    With all this set up, we finally see how we can estimate $\sigma_K(t,\mathcal{C})$ using lattice theory.
    \begin{corollary}\label{C:fundDoms}
    Let $D\subseteq H$ be a fundamental domain of the unit lattice $\Gamma$, and let $v\in \lsp$, such that $\Nl (v) = 1$. Then for all $t>0$, we have
		$$
		\sigma_K(t^n,\mathcal{C})= \frac{1}{\omega_K} \vert A_t \cap \F(D,\left(0,1\right],v ) \vert.
		$$
    \end{corollary}
    
    \begin{proof}
        Recall from the start of this section that $\sigma_K(t^n,\mathcal{C})$ is equal to the number of orbits in $E(\mathfrak{a},t^n)=\{\alpha \in \mathfrak{a} \mid 0< \vert N_{K/\Q}(\alpha)\vert \leq N(\mathfrak{a}) t^n \}$ under the action of $\mathcal{O}_K^\times$.
        By part \eqref{i:normFund} of Lemma \ref{L:fundsProperties} and the definition of $A_t$, we see that $\frac{1}{t N(\mathfrak{a})^{1/n}} \phi$ gives a bijection between the sets $E(\mathfrak{a},t^n)$ and $A_t \cap \F(H,\left(0,1\right],v )$.
        Moreover, multiplying elements of $\mathfrak{a}$ with a unit $\varepsilon\in \mathcal{O}_K$ corresponds to the action of $\Psi_{\phi(\varepsilon)}$ on $A_t$. Our statement now follows from part \eqref{i:unitFund} of Lemma \ref{L:fundsProperties}.
    \end{proof}

    \begin{remark}
        It might seem more natural to state that
        $$
        \sigma_K(t,\mathcal{C})=\frac{1}{\omega_K}\vert \phi(\mathfrak{a}) \cap \F(D,\left(0,N(\mathfrak{a}) t \right], v)  \vert,
        $$
        since this corresponds more directly to the definition of $E(\mathfrak{a},t)$. However, with our normalisation, the domain $\F(D,\left(0,1\right],v )$ becomes independent of both $\mathfrak{a}$ and the parameter $t$. This is very convenient, since we will spend a lot of time analysing this domain.
    \end{remark}

    Now we show that we can replace the fundamental domain $D$ in Corollary \ref{C:fundDoms} with a smaller domain $\tilde D \subseteq H$ by invoking Schmidt's partition trick. 

	\begin{lemma}[Schmidt's partition trick]\label{L:Schmidt}
		Let $\Lambda \subseteq \mkw$ be a lattice. Let $D,\tilde D \subseteq H$ be subsets together with a set $\mathcal{S}$ and elements $y_s \in H$ for every $s\in \mathcal{S}$, such that $D=\bigsqcup_{s\in S} \tau_{y_s} \tilde{D}$. Let $\beta_s \in \mkw$, such that $\ell(\beta_s)=-y_s$. Then
		$$
		\vert \Lambda \cap \F(D,I,v) \vert = \sum_{s\in \mathcal{S}} \vert \Psi_{\beta_s} \Lambda \cap \F(\tilde{D},I,v) \vert .
		$$
	\end{lemma}

    \begin{remark}
        This trick was developed by Wolfgang Schmidt \cite{schmidt_northcotts_1995} for the real quadratic case. Widmer \cite{Widmer}  generalised this to partitioning the fundamental parallelepiped of the unit lattice in any degree, which was also used by Debaene \cite{Korneel}. We now observe that we can do the same trick for a larger range of domains $D\subseteq H$.
    \end{remark}

	\begin{proof}
		In the setting above, it is clear that we have
		$$
		\F(D,I,v)=\bigsqcup_{s\in S} \F(\tau_{y_s} \tilde{D}, I,v).
		$$
		Using Lemma \ref{L:TransDia}, we can write $\F(\tau_{y_s} \tilde{D}, I,v)=\Psi_{\beta_s}^{-1} \F( \tilde{D}, I,v)$, and thus
		\begin{alignat*}{2}
			 \vert \Lambda \cap \F(D,I,v) \vert &= \sum_{s\in \mathcal{S}} \vert \Lambda \cap \F(\tau_{y_s} \tilde{D},I,v) \vert \\
			 &= \sum_{s\in \mathcal{S}} \vert \Lambda \cap \Psi_{\beta_s}^{-1} \F(\tilde{D},I,v) \vert \\
			 &= \sum_{s\in \mathcal{S}} \vert \Psi_{\beta_s} \Lambda \cap  \F(\tilde{D},I,v) \vert .
		\end{alignat*}
	\end{proof}
	
	\subsection{Specific choices}\label{Ss:SpecChoice}
	We choose an LLL-reduced basis $l_1,\dots,l_r$ of the unit lattice $\Gamma$ and consider the Gram-Schmidt orthogonalised basis $e_1,\dots , e_r$. By Lemma \ref{L:LLL}, we thus have $\|l_i\| \leq 2^{\frac{i-1}{2}} \|e_i \|$.
	
	For some $c\geq 1$, satisfying $2\pi c \in \mathbb{N}$, we define
	$$\delta_i :=  \lceil c \|e_i\| \rceil $$
	for $i=1,\dots,r$.
    
    \begin{remark}
        We will make a concrete choice for $c$ in Lemma \ref{L:ugly}, but until then it is not clear, why the choice we make there is good. Therefore, we will not assume this specific choice until Section \ref{S:main}. To be able to make some slight simplifications of some terms (for example in the proof of Corollary \ref{C:LipConstF}), it will be useful to have some lower bound on the size of $c$. This is why we, for now, work with $c\geq 1$. Moreover, it will be very convenient later to assume that $2\pi c \in \mathbb{N}$, so that we can divide the interval $\left[0, 2\pi \right)$ evenly into subintervals of length $c^{-1}$ (see Definition \ref{D:G}).
    \end{remark}
	
	Collecting the $\delta_i$ into $\delta =(\delta_1,\dots, \delta_r) \in \mathbb{Z}^r$, we define
	$$D_{\delta} := \left\{ \sum_{i=1}^r \lambda_i  e_i \biggm \vert \lambda_i \in \left[ -\frac{1}{2\delta_i}, \frac{1}{2\delta_i} \right)\right\}.$$
	
	Moreover, we let
	$\mathcal{S}=\prod_{i=1}^{r} ( \mathbb{Z} \cap \left[0,\delta_i \right))$,
	and, for $s\in \mathcal{S}$, we define
	$y_s := \sum_{i=1}^{r} \frac{s_i}{\delta_i} e_i \in H$. 
	Then we have that
	$$\bigsqcup_{s\in S} \tau_{y_s} D_{\delta} = \left\{ \sum_{i=1}^r \lambda_i  e_i \biggm \vert \lambda_i \in \left[-\frac{1}{2\delta_i} , 1 -\frac{1}{2\delta_i} \right) \right\}$$
	is a fundamental domain of $\Gamma$.
    We let $\beta_s:=\left( \exp\left(\frac{-y_j}{n_j} \right)\right)_{j=1}^{r_1+r_2}\in \mkw$. Notice that $\ell(\beta_s)=-y_s$.
	Finally, we let $\nu:=\left(\frac{1}{r_1+r_2}\right)_{i=1}^{r_1+r_2}\in \lsp$ satisfying $\Nl(\nu)=1$ and $\|\nu\|=\frac{1}{\sqrt{r_1+r_2}}$.
	
	Now Corollary \ref{C:fundDoms} and Lemma \ref{L:Schmidt} together yield the following result.
	\begin{proposition}\label{P:SigLat}
        For all $t>0$, we have
		$$
		\sigma_K(t^n,\mathcal{C}) = \frac{1}{\omega_K} \sum_{s \in \mathcal{S}} \vert (\Psi_{\beta_s} A_t) \cap \F( D_{\delta},\left(0,1\right],\nu) \vert.
		$$
	\end{proposition}
	
	\begin{lemma}\label{L:SBound}
    The set $\mathcal{S}$ satisfies
		$$
        c^r \sqrt{r_1+r_2} R_K \leq \vert \mathcal{S} \vert \leq \exp\left(\frac{14 n 2^{\frac{r}{2}}}{c} \right)  c^r \sqrt{r_1+r_2} R_K.
        $$
	\end{lemma}
	
	\begin{proof}
        For the lower bound, we observe that $\delta_i\geq c\|e_i\|$, and hence we obtain
        $$
        \vert S \vert = \prod_{i=1}^r \delta_i \geq \prod_{i=1}^r c\|e_i\| =c^r \vol(\Gamma)=c^r \sqrt{r_1+r_2} R_K.
        $$
		To show the upper bound, we use that we chose $l_1,\dots,l_r$ LLL-reduced, i.e., we have $\|e_i\| \geq 2^{\frac{1-i}{2}} \|l_i\|$ by Lemma \ref{L:LLL}. Recall that $l_i = \ell\circ \phi(\varepsilon)$ for a unit $\varepsilon \in \mathcal{O}_K$, which is not a root of unity. We have $\|l_i\| \geq \max_j \log(\vert \sigma_j(\varepsilon)\vert)$, and moreover, $\max_j \vert \sigma_j(\varepsilon) \vert \geq 2^{1/4n}$ is due to Dimitrov \cite{Modulus}. Hence,
		$$
		\|e_i\| \geq 2^{\frac{1-i}{2}} \| l_i \| \geq 2^{\frac{1-i}{2}} \frac{\log(2)}{4n}.
		$$
        With this we can bound
        $\delta_i = \lceil c \|e_i\| \rceil$ from above,
		$$\delta_i \leq c \|e_i\|+1\leq \left( 1+ \frac{2^{\frac{i-1}{2}} 4n}{ c  \log(2) } \right) c \|e_i\|.$$
		
        The last step is to bound the product $\prod_{i=1}^{r} \left( 1+ \frac{2^{\frac{i-1}{2}} 4n}{ c  \log(2) } \right)$. Taking the logarithm we get
		$$
			\sum_{i=1}^r \log \left( 1+ \frac{2^{\frac{i-1}{2}} 4n}{ c  \log(2) } \right) \leq \sum_{i=1}^r \frac{2^{\frac{i-1}{2}} 4n}{ c  \log(2) }
			= \frac{ 4n}{ c  \log(2) } \frac{2^{\frac{r}{2}} - 1}{ \sqrt{2} - 1}
			\leq \frac{14 n 2^{\frac{r}{2}}}{c},
		$$
	and together with $\prod_{i=1}^r \|e_i\| = \vol \Gamma = \sqrt{r_1+r_2} R_K$ this finishes the proof.
	\end{proof}
	
	\begin{remark}\label{R:Sbound}
		This lemma is a good example of why it is important that we generalised the Schmidt partition in the way we did. If we had used the $l_1,\dots,l_r$ to define the domain $\F( D_{\delta},\left(0,1\right],\nu)$, we would not have $\prod_{i=1}^r \|l_i\| = \vol \Gamma$. Therefore, when estimating the size of $\mathcal{S}$, we would have to compensate with an orthogonality defect, i.e., a factor of $\frac{r^{\frac{3}{2} r}}{(2\pi)^{\frac{r}{2}}}$ (compare to Lemma 9 from \cite{Korneel}). Moreover, being able to choose the value $c$ means that we can get a very precise estimation for the size of $\mathcal{S}$.
	\end{remark}

	\section{Lipschitz class}\label{S:Lipschitz}
	
	Now we turn our attention to our fundamental domain $\F ( D_{\delta}, \left( 0,1\right], \nu )$. our goal is to compute the Lipschitz class of the boundary $\partial \F ( D_{\delta}, \left( 0,1\right], \nu )$, which we need to apply our counting method, Proposition \ref{P:LatLipCounting}.

    \subsection{An isomorphism}
    
	To make the analysis easier, we would like to be able to go back and forth between our two important spaces, the Minkowski space $\mkw$ and the logarithmic space $\lsp$. However, the map $\ell: \mkw^\times \to \lsp$ is not bijective. To remedy this situation, we define the following objects.
	
	We let
	$$
	\mkw^+ := (\R^+)^{r_1} \times \C^\times,
	$$
    which is a subgroup of $\mkw^\times$, and we let
	$$
	\F(a,b) := \F ( D_{\delta}, \left( a,b\right], \nu ) \cap \dmk.
	$$
	\begin{remark}\label{R:+}
	    Observe that  $\F ( D_{\delta}, \left( a,b\right], \nu )$ consists of $2^{r_1}$ copies of $\F(a,b)$, i.e., $\vol \F ( D_{\delta}, \left( a,b\right], \nu )=2^{r_1}\vol \F(a,b)$. If $\partial\F(a,b)$ is of Lipschitz class $\operatorname{Lip}(N,M,L)$, then $\partial \F ( D_{\delta}, \left( a,b\right], \nu )$ is of Lipschitz class $\operatorname{Lip}(N,2^{r_1}M,L)$.
	\end{remark}

    We also want to keep track of the argument of the complex coordinates in $\mkw$. We therefore let 
    $$\T:=(\R/2\pi\Z)^{r_2}$$
    and refine $\ell$ to the map
	\begin{alignat*}{2}
		f : \dmk &\rightarrow \lsp \times \T, \\
		(x_i) &\mapsto (\ell(x),\arg(x_{r_1+1}),\dots,\arg(x_{r_1+r_2})).
	\end{alignat*}

	The map $f$ is, in fact, an isomorphism of groups with inverse
	\begin{alignat*}{2}
		g : \lsp \times \T &\rightarrow  \dmk, \\
		(x  , \theta) &\mapsto \left( \exp(x_1), \dots , \exp(x_{r_1}),\exp\left(\frac{x_{r_1+1}}{2} + i \theta_1 \right) ,\dots , \exp\left(\frac{x_{r_1+r_2}}{2} + i \theta_{r_2}\right) \right).
	\end{alignat*}
	We let $\E(a,b):=\E ( D_{\delta}, \left( a,b\right], \nu )$ and
    $$
    \G(a,b):=\E (  a,b) \times \T,
    $$
   such that  $\F(a,b)$ and $\G(a,b)$ are in bijection under $f$ and $g$, i.e., $\G(a,b)=f(\F(a,b))$ and $\F(a,b)=g(\G(a,b))$.

	\begin{lemma}\label{L:Jac}
		For $(x,\theta)\in \lsp \times \T$, the Jacobian $Dg(x,\theta)$ satisfies
		$$\vert \det (Dg(x,\theta)) \vert = 2^{-r_2}\exp\left(\sum_{i=1}^{r_1+r_2} x_i \right) = 2^{-r_2} \exp( \Nl (x))$$
		and
		$$\| Dg(x,\theta)\|_2 = \max_i \exp\left(\frac{x_i}{n_i}\right) .$$
        Here, $\|\cdot\|_2$ denotes the spectral norm.
	\end{lemma}
    \begin{proof}
        
	The map $g$ has Jacobian
	\begin{equation*}
		Dg(x,\theta)= \begin{pmatrix} 
			  e^{x_1} & & 0 & & & & & &\\
			& \hspace{-6pt} \ddots \hspace{-4pt} & & & &  & & &\\
			0 & &  e^{x_{r_1}} & & & & & &\\
			& & & \frac{\cos(\theta_1)}{2} e^{\frac{x_{r_1+1}}{2}} & & & \sin(\theta_1) e^{\frac{x_{r_1+1}}{2}} & &\\
			& & & \frac{\sin(\theta_1)}{2} e^{\frac{x_{r_1+1}}{2}} &  & 0 & -\cos(\theta_1) e^{\frac{x_{r_1+1}}{2}}& & 0\\
			&  & & & \hspace{-8pt} \ddots \hspace{-2pt} & & & \hspace{-8pt} \ddots \hspace{-2pt} &\\
			& & & 0 & & \frac{\cos(\theta_{r_2})}{2} e^{\frac{x_{r_1+r_2}}{2}} &  0 & & \sin(\theta_{r_2}) e^{\frac{x_{r_1+r_2}}{2}}\\
			& & & & & \frac{\sin(\theta_{r_2})}{2} e^{\frac{x_{r_1+r_2}}{2}} & & & -\cos(\theta_{r_2}) e^{\frac{x_{r_1+r_2}}{2}}\\
		\end{pmatrix},
	\end{equation*}
		which we can factor as $Dg(x,\theta)=A B$, where
		\begin{equation*}
			A= \begin{pmatrix} 
				1& & 0 & & & & & &\\
				& \ddots & & & & & & &\\
				 0 & & 1 & & & & & &\\
				& & & \cos(\theta_1) & & & \sin(\theta_1)& &\\
				& & & \sin(\theta_1) & & 0 & -\cos(\theta_1) & & 0\\
				& & & & \ddots & & & \ddots &\\
				& & &  0 & & \cos(\theta_{r_2})& 0 & & \sin(\theta_{r_2})\\
				& & & & & \sin(\theta_{r_2}) & & & -\cos(\theta_{r_2}) \\
			\end{pmatrix}
		\end{equation*}
	is orthogonal and 
	\begin{equation*}
		B= \begin{pmatrix} 
			\exp({x_1}) & & & & & & & &\\
			& \hspace{-7pt}\ddots & & & & & & &\\
			& & \hspace{-7pt}\exp({x_{r_1}}) & & & & 0 & &\\[7pt]
			& & & \hspace{-7pt} \frac{1}{2} \exp\left({\frac{x_{r_1+1}}{2}}\right) & & &  & &\\
			& & & & \hspace{-14pt}\ddots & & & &\\
			& & & & & \hspace{-14pt}\frac{1}{2} \exp \left({\frac{x_{r_1+r_2}}{2}}\right) & & & \\[7pt]
			& & 0 & & & & \hspace{-7pt} \exp\left( {\frac{x_{r_1+1}}{2}}\right) & & \\
			& & & & & & & \hspace{-10pt} \ddots  & \\
			& & & & & & & & \hspace{-10pt}\exp \left({\frac{x_{r_1+r_2}}{2}}\right) 
		\end{pmatrix}
	\end{equation*}
    is in diagonal form. The statement readily follows.
	\end{proof}

	We can use this result to show the following identity.
	\begin{proposition}\label{P:volF}
		We have $\vol(\F(0,1))=\frac{\pi^{r_2} R_K}{\vert \mathcal{S} \vert} $.
	\end{proposition}

    \begin{remark}
		  With this, we can already see, how the value $\kappa_K$ arises geometrically. We have
        $\sigma_K(t^n)=\sum_{\mathcal{C}} \sigma_K(t^n,\mathcal{C})$, and, by Corollary \ref{C:fundDoms}, we know that $\sigma_K(t^n,\mathcal{C})=\frac{1}{\omega_K} \vert A_t \cap \F(D,\left(0,1\right],\nu ) \vert$, where $A_t$ is a lattice of volume $\frac{\sqrt{\Delta_K}}{2^{r_2} t^n}$. 
        Using the above proposition, we obtain
        $$\vol \F(D,\left(0,1\right],\nu)=2^{r_1}\vert \mathcal{S} \vert \vol \F(0,1) =2^{r_1}\pi^{r_2} R_K,$$
        so we can expect $\sigma_K(t^n)$ to grow like
        $$
        \frac{h_K}{\omega_K}\frac{\vol \F(D,\left(0,1\right],\nu )}{\vol A_t}=\frac{2^{r_1+r_2} \pi^{r_2} h_K R_K}{\omega_K \sqrt{\Delta_K}}t^n = \kappa_K t^n.
        $$
	\end{remark}
    
	\begin{proof}[Proof of Proposition \ref{P:volF}]
    We can easily calculate the volume of $D_{\delta}$. This is because $D_{\delta}$ is an $r$-dimensional cuboid with sidelengths $\frac{\|e_i\|}{\delta_i}$, and, since the $e_i$ span an orthogonal fundamental domain of the unit lattice, we have
		$$\vol(D_{\delta})=\frac{\prod \|e_i\|}{\prod \delta_i}=\frac{\sqrt{r_1+r_2} R_K}{\vert \mathcal{S} \vert}.$$
		Now we can use the integral transformation formula and Lemma \ref{L:Jac} to compute
		\begin{alignat*}{2}
			\vol(\F(0,1))&=\vol(g(\G(0,1))) \\
			&= \int_{g(\G(0,1))} 1 \;\mathrm{d}x \\
			&= \int_{\G(0,1)} \vert \det(Dg(y)) \vert \; \mathrm{d}y\\
			&= (2\pi)^{r_2} \int_{\E(0,1)} 2^{-r_2} \exp(\Nl(y)) \; \mathrm{d}y\\
			&\overset{(\ast)}{=}\pi^{r_2}  \int_{-\infty}^0 \frac{1}{\sqrt{r_1+r_2}} \int_{D_{\delta}}   \exp(\Nl(h+\lambda  \nu)) \; \mathrm{d}h \; \mathrm{d}\lambda\\
			&=\pi^{r_2}\int_{-\infty}^0 \frac{1}{\sqrt{r_1+r_2}} \vol(D_{\delta})  \exp(\lambda) \; \mathrm{d}\lambda\\
			&= \pi^{r_2} \frac{\sqrt{r_1+r_2} R_K}{ \sqrt{r_1+r_2}\vert \mathcal{S} \vert}\\
			&= \frac{\pi^{r_2} R_K}{\vert \mathcal{S} \vert}.
		\end{alignat*}
		
		In step $(\ast)$ we make the transformation  $ (h,\lambda ) \mapsto h + \lambda \nu=y$, which has determinant $\|\nu\|=\frac{1}{\sqrt{r_1+r_2}}$.
	\end{proof}

	\begin{lemma}\label{L:gLip}
		For all $x,y \in \G(0,1)$, we have
		$$
		\|g(x)- g(y)\| \leq \exp \left( \frac{\sqrt{r}}{2c} \right)  \| x-y \|.
		$$
	\end{lemma}
    \begin{proof}
        Let $\gamma:[0,\|x-y\|] \to \lsp$ be the isometric parametrisation of the direct path from $x$ to $y$. Then $g\circ \gamma$ is a path from $g(x)$ to $g(y)$ and the distance $\|g(x)- g(y)\|$ is therefore bounded by the arc-length of the path $g\circ \gamma$, i.e.,
        $$
        \|g(x)- g(y)\| \leq \int_0^{\|x-y\|} \|D(g\circ\gamma)(u)\| \,\mathrm{d}u =  \int_0^{\|x-y\|} \|D(g)(\gamma(u))  D \gamma(u)\| \,\mathrm{d}u.
        $$
        Since $\gamma$ is isometric, $D \gamma(u)$ has norm one. The last integrand is therefore bounded by the spectral norm $\| Dg(\gamma(u))\|_2$. By construction, $\G(0,1)$ is convex, so $\gamma(u)$ is inside this set, for all $u\in [0,\|x-y\|]$.

        We can thus write $\gamma(u)=(h_u+\mu_u\nu,\theta_u)$ with $h_u \in  D_{\delta}$ and $ \mu_u\leq0$. From Lemma \ref{L:Jac} we know that the spectral norm only depends on the coordinates of $h_u+\mu_u\nu$, i.e.,
        $$
        \| Dg(\gamma(u))\|_2 = \max_i\exp\left( \frac{h_{u,i}+\mu_u \nu_i}{n_i}\right).
        $$
        Since $\mu_u\leq 0$ and all coordinates of $\nu$ are positive, the above is bounded by $\max_i \exp\left(\frac{h_{u,i}}{n_i}\right)\leq \exp \| h_u \|$.

        Now, since $D_{\delta}$ is an $r$-dimensional cuboid centred at zero with sidelengths less than or equal to $c^{-1}$, we get the bound
        $$\| Dg(\gamma(u))\|_2 \leq \exp\left(\frac{\sqrt{r}}{2c}\right),$$
        which gives the desired result.
    \end{proof}

    \subsection{Lipschitz class computation}\label{Ss:LipComp}

    We now show that $\partial  \F( 0,1)$ has a certain Lipschitz class, so that we can apply Proposition \ref{P:LatLipCounting}. As a first step, we will focus our analysis on a small subset of $\F( 0,1)$. Then we can use the correspondence from Lemma \ref{L:TransDia} to obtain a result about the whole domain $\F(0,1)$.

    Recall that $D_{\delta}$ is an $r$-dimensional cuboid with sidelengths less than or equal to $c^{-1}$ and that the vector $\nu =\left(\frac{1}{r_1+r_2}\right)_{i=1}^{r_1+r_2}$ is orthogonal to $D_{\delta}$ in $\lsp$.
    We can therefore find a parameter $0<w<1$, such that
    $\E( w,1)$ is an $(r_1+r_2)$-dimensional cuboid with sidelengths less than or equal to $c^{-1}$.

    An element $x\in\E( w,1)$ can be written as $x=h+\mu\nu$ with $\exp(\mu)\in\left(w,1 \right]$. We want $\mu$ to satisfy $\mu\in \left(-\frac{1}{c\|\nu\|},0 \right]$. Hence, we take $w=\exp \left(-\frac{1}{c\|\nu\|}\right)$.

    We will also divide $\T$ into cubes of sidelength $c^{-1}$. For this, it is very convenient to choose $c$, such that $2\pi c\in \mathbb{N}$.
	
    \begin{definition}\label{D:G}
        We let  $\mathcal{J}:=\prod_{j=1}^{r_2} ( \mathbb{Z} \cap \left[0,2\pi c\right) )$, and, for $\gamma\in \mathcal{J}$, we let
	$$\G_{\gamma}(a,b):=\E(a,b) \times \prod_{j=1}^{r_2} \left[\frac{\gamma_j}{c},\frac{\gamma_j + 1}{c} \right) \subseteq\G(a,b).$$
    \end{definition}
     
    We can easily see that $\partial \G_{\gamma}(w,1)$ is of Lipschitz class $\operatorname{Lip}(n-1,2n,c^{-1})$.
	
	Moreover,
	$$\G(a,b) = \bigsqcup_{\gamma \in \mathcal{J}} \G_{\gamma}(a,b) .$$

    We can use Lemmas \ref{L:gLip} and \ref{L:TransDia} to determine a Lipschitz class for $\partial \F(w^{k+1},w^k)$ for $k\geq0$. To get a result for the boundary of the full domain $\F(0,1)$, the following glueing lemma will be very useful.
    
	\begin{lemma}\label{L:fuseLip}
		Let $\varphi_i : [0,\ell_i] \times [ 0,1]^{N-1} \to \R^m$, for $i=1,2$, be maps, such that
		$\|\varphi_i (x) - \varphi_i(y)\| \leq L_i \|x-y\|$ with $L_1\geq L_2 > 0$, and $\varphi_1(\ell_1,x_2,\dots,x_N) = \varphi_2(0,x_2,\dots,x_N)$.
		Then there exists a map 
		$$\tilde{\varphi} : \left[0,\ell_1 + \frac{L_2}{L_1} \ell_2 \right] \times [ 0,1]^{N-1} \to \R^m,$$
		such that $\operatorname{im}(\tilde{\varphi}) = \operatorname{im}(\varphi_1) \cup \operatorname{im}(\varphi_2)$, and $\| \tilde\varphi (x) - \tilde\varphi (y)\| \leq L_1 \|x-y\|$.
	\end{lemma}
	
	\begin{proof}
		We define 
		$$
		\tilde{\varphi} (x_1,\dots,x_N) = \begin{cases}
			\varphi_1(x_1,\dots,x_N) &\text{if } x_1\leq \ell_1, \\
			\varphi_2\left(\frac{L_1}{L_2}(x_1 - \ell_1),x_2,\dots,x_N\right) &\text{if } x_1\geq \ell_1.
		\end{cases}
		$$
		Observe that this is welldefined for $x_1=\ell_1$, since $\varphi_1(\ell_1,x_2,\dots,x_N) = \varphi_2(0,x_2,\dots,x_N)$. The property $\operatorname{im}(\tilde{\varphi}) = \operatorname{im}(\varphi_1) \cup \operatorname{im}(\varphi_2)$ is trivial. To show the Lipschitz condition, we distingish between three cases for $x,y \in \left[0,\ell_1 + \frac{L_2}{L_1} \ell_2 \right] \times [ 0,1]^{N-1}$.
		
		\textit{Case 1:} $x_1,y_1 \leq \ell_1$. Clear, since $\tilde{\varphi}$ agrees with $\varphi_1$.
		
		\textit{Case 2:} $x_1,y_1 \geq \ell_1$. We have
		\begin{alignat*}{2}
			\| \tilde\varphi (x) - \tilde\varphi (y) \| &= \left\|\varphi_2\left(\frac{L_1}{L_2}(x_1 - \ell_1),x_2,\dots,x_N\right) - \varphi_2\left(\frac{L_1}{L_2}(y_1 - \ell_1),y_2,\dots,y_N\right) \right\| \\
			& \leq L_2 \left\|\left(\frac{L_1}{L_2}(x_1 - y_1),x_2-y_2,\dots,x_N-y_N\right) \right\| \\
			& = \sqrt{L_1^2 (x_1-y_1)^2 + L_2^2 (x_2-y_2)^2  + \dots L_2^2 (x_N-y_N)^2} \\
			& \leq \sqrt{L_1^2 (x_1-y_1)^2 + L_1^2 (x_2-y_2)^2  + \dots L_1^2 (x_N-y_N)^2} \\
			&= L_1 \|x-y\|.
		\end{alignat*}
		\textit{Case 3:} $x_1 \leq \ell_1 \leq y_1$. We can find a point $z \in \left[0,\ell_1 + \frac{L_2}{L_1} \ell_2 \right] \times [ 0,1]^{N-1}$, such that $z_1 = \ell_1$ and $\|x-y\| = \|x-z\| + \|z-y\|$. Now the statement follows by applying Case 1 to $x$ and $z$, and Case 2 to $z$ and $y$.
	\end{proof}

	\begin{lemma}\label{L:mapPhi} There is a map $\varphi : \left[ 0,2 \sqrt{r_1+r_2} c + 1 \right] \times [0,1]^{n-1} \to \mkw$, such that $\operatorname{im}(\varphi) = \overline{g\left(\G_{\gamma}(0,1) \right)}$, and, for all $x,y\in \left[ 0,2 \sqrt{r_1+r_2} c + 1 \right] \times [0,1]^{n-1}$, we have
		$$
		\|\varphi (x) - \varphi(y)\| \leq  \exp \left( \frac{\sqrt{r}}{2c} \right) c^{-1} \|x-y\|.
		$$
	\end{lemma}

	\begin{proof}
		 Let $b_1,\dots,b_{r_2}$ be the canonical basis of $\T$ (that is $b_1=(1,0,\dots,0)\in \T$, etc.). We define the map
		\begin{alignat*}{2}
			\psi_0 : [0,1]^n &\to \lsp \times \T \\
			(x_1,\dots, x_n) &\mapsto  \sum_{i=1}^{r}\frac{x_i - \frac{1}{2}}{\delta_i} e_i  -  \frac{x_{r_1+r_2}}{c\|\nu\|} \nu + \sum_{j=r_1+r_2+1}^{n} \frac{x_{j} + \gamma_j}{c} b_j. 
		\end{alignat*}
		Observe that the image of $\psi_0$ is exactly $\overline{\G_{\gamma}(w,1)}$. 
        
        Since $e_1,\dots,e_r,v,b_{r_1+r_2+1},\dots,b_{n}$ are pairwise orthogonal in $\lsp\times\T$, we have $\| \psi_0 (x) - \psi_0(y) \| \leq c^{-1} \|x-y\|$ for all $x,y\in [0,1]^n$.
		
		Using Lemma \ref{L:gLip}, we get that that the map $\varphi_0 := g\circ \psi_0$ satisfies $$\|\varphi_0(x)- \varphi_0(y)\| \leq \exp \left( \frac{\sqrt{r}}{2c} \right) c^{-1} \| x-y \|.$$

		Next we observe that we can extend our definition of the translation maps $\tau_y:\lsp \to \lsp$ with $y\in\lsp$ from Lemma \ref{L:TransDia} to $\tau_{(y,\theta)}:\lsp\times\T \to \lsp \times \T$ with $(y,\theta)\in \lsp \times \T$, and that this gives rise to a commutative diagram
		\[
		\xymatrix{\lsp \times \T \ar[d]_-{g} \ar[r]^-{\tau_{(y,\theta)}} & \lsp \times \T \ar[d]^-{g}     
			 \\
			\mkw \ar[r]_-{\Psi_{g(y,\theta)}} & \mkw}
		\]
        (compare with Lemma \ref{L:TransDia}).

        We have 
        $$\G_{\gamma}(w^{k+1},w^k)=\tau_{(\log(w^k)\nu,0)}\G_{\gamma}(w,1)=\tau_{(\log(w)\nu,0)}^k\G_{\gamma}(w,1).$$
        
		Thus, by letting $\psi_{k} = \tau_{(\log(w)\nu,0)}^{k} \circ \psi_1$, for $k=1,2,\dots$, we get that $$\overline{\G_{\gamma}(0,1)} = \bigcup_{k=0}^\infty \operatorname{im}(\psi_{k}).$$
		Now we let $\varphi_k=g\circ \psi_k = \Psi_{g(\log(w)\nu,0)}^{k} \circ \varphi_0$.
        Since $g(\log(w)\nu,0)=\left(w^{\frac{1}{n_i(r_1+r_2)}}  \right)_{i=1}^{r_1+r_2}$, the linear map $\Psi_{g(\log(w)\nu,0)}$ is given by
        $$
        (x_1,\dots,x_{r_1+r_2})\mapsto \left( w^{\frac{1}{r_1+r_2}} x_1, \dots , w^{\frac{1}{r_1+r_2}} x_{r_1}, w^{\frac{1}{2(r_1+r_2)}} x_{r_1+1} ,\dots , w^{\frac{1}{2(r_1+r_2)}} x_{r_1+r_2}  \right).
        $$
        Thus,
		$$\|\varphi_k(x)- \varphi_k(y)\| \leq w^{\frac{k}{2(r_1+r_2)}} \exp \left( \frac{\sqrt{r}}{2c} \right) c^{-1} \| x-y \|, \quad \forall k\geq 0.$$
		We can now apply Lemma \ref{L:fuseLip} to iteratively glue the maps $\varphi_0,\varphi_1,\varphi_2,\dots$ together to a single map $\tilde{\varphi}$. This map will be defined on
		$$
		\left[0, \sum_{k=0}^\infty w^{\frac{k}{2(r_1+r_2)}} \right) \times [0,1]^{n-1}
		$$
		and satisfy the same Lipschitz-condition as $\varphi_0$. We calculate
		$$
			\sum_{k=0}^\infty w^{\frac{i}{2(r_1+r_2)}} = \frac{1}{1- w^{\frac{1}{2(r_1+r_2)}}}
			= \frac{1}{1- \exp \left(\frac{-1}{2 \sqrt{r_1+r_2} c} \right) }
			 \leq 2 \sqrt{r_1+r_2} c + 1,
		$$
        where, in the last step, we use that $\frac{1}{1-\exp(-1/x)}\leq x+1$.
        Hence, we obtain $\varphi :\left[ 0,2 \sqrt{r_1+r_2} c + 1 \right) \times [0,1]^{n-1} \to \mkw$ covering all of $g\left(\G_{\gamma}(0,1) \right)$.
        We then extend this map continuously to $\left[ 0,2 \sqrt{r_1+r_2} c + 1 \right] \times [0,1]^{n-1}$, by letting $\varphi ( 2 \sqrt{r_1+r_2} c + 1 , x_2 , \dots , x_n) = 0$.
	\end{proof}
	
	\begin{corollary}\label{C:LipConstF}
		The set
		$\partial  \F( D_{\delta},\left(0,1\right],\nu)$ is of Lipschitz class $\operatorname{Lip}\left(n-1,M,L \right)$
		 with
		$$M\leq 6n \sqrt{r_1+r_2} c  \left(2\pi c\right)^{r_2} 2^{r_1}\quad \text{and} \quad L=\exp \left( \frac{\sqrt{r}}{2c} \right) c^{-1}.$$
	\end{corollary}

	\begin{proof}
		We continue working with the map $\varphi$ from Lemma \ref{L:mapPhi}. By partitioning the interval $[0, 2 \sqrt{r_1+r_2} c + 1]$ into $m=\lceil 2 \sqrt{r_1+r_2} c + 1 \rceil \leq 3 \sqrt{r_1+r_2} c $ parts, we obtain maps
		$xi_1,\dots,xi_m : [0,1]^{n} \to \mkw$ covering exactly $\overline{g\left(\G_{\gamma}(0,1) \right)}$ and satisfying $\|xi_i(x)-xi_i(y)\| \leq L\|x-y\|$. 
		
		Since $$\F(0,1) = g(\G(0,1))= \bigsqcup_{\gamma \in \mathcal{J}} g\left(\G_{\gamma}(0,1) \right),$$
		
		we can cover $\overline{\F(0,1)}$ with $m \left(2\pi c\right)^{r_2}$ such maps --- all with Lipschitz-constant $L=\exp \left( \frac{\sqrt{r}}{2c} \right) c^{-1}$.
		
		Thus, $\partial \F(0,1)$ can be covered by $2n m \left(2\pi c\right)^{r_2}$ maps $[0,1]^{n-1} \to \mkw$ with the same Lipschitz-constant, which gives the desired result, by Remark \ref{R:+}.
	\end{proof}
	
	\section{Minima}\label{S:Minima}

    The last ingredient, we need to apply Proposition \ref{P:LatLipCounting} and obtain our main result, is an upper bound for the covering radius $r(\Psi_{\beta_s} A_t)$. By Lemma \ref{L:RadMin}, we can do this by giving a lower bound for $\lambda_1(\Psi_{\beta_s} A_t)^{1-n}$.

    Recall that $A_t = \frac{1}{t N(\mathfrak{a})^{1/n}} \phi(\mathfrak{a})$, where $\mathfrak{a}\subseteq\mathcal{O}_K$ is an integral ideal in the class $\mathcal{C}^{-1}$. Moreover, the elements $\beta_s \in \mkw$ were chosen to satisfy $\ell(\beta_s)= - y_s =  - \sum_{i=1}^{r} \frac{s_i}{\delta_i} e_i$.
    The vectors $e_i\in H$, the scalars $\delta_i$, and the set $\mathcal{S}\subseteq\Z^r$, from which we take our elements $s$, were defined in Subsection \ref{Ss:SpecChoice}.

    We present two different bounds for our minima. The first will only look at one lattice $\Psi_{\beta_s} A_t$ at a time, while the second approach bounds the sum
    $\sum_{s\in \mathcal{S}} \lambda_1(\Psi_{\beta_s} A_t)^{1-n}$.
    
	\begin{lemma}\label{L:MinBound}
		There is an integral ideal $\mathfrak{b} \in \mathcal{C}$, such that
		$$ 
		\lambda_1(\Psi_{\beta_s} A_t) \geq \frac{\sqrt{n}}{\sqrt{2}}  N(\mathfrak{b})^{1/n}  t^{-1}.
		$$
	\end{lemma}
\begin{remark}
	The proof of this Lemma is based on a similar calculation from \cite{Preda}. Here Mih\u{a}ilescu shows the upper bound
	$$
	d< \frac{n^{\frac{3n+1}{2}}  \vol(A_1)}{N(\mathfrak{b})^{\frac{n-1}{n}}  t},
	$$
	where $d$ is the length of the diagonal of the fundamental parallelepiped of $A_t$. We have adapted the proof to work for $\lambda_1(\Psi_{\beta_s} A_t)$.
\end{remark}
    
	\begin{proof}[Proof of Lemma \ref{L:MinBound}]
		There exist $\alpha \in \mathfrak{a}$, such that $N(\mathfrak{a})^{-1/n} t^{-1} \Psi_{\beta_s}\phi(\alpha)$ has norm $\lambda_1(\Psi_{\beta_s} A_t)$. We now use the AM--GM inequality and the fact that $\beta_s$ has norm one to get
		\begin{alignat*}{2}
			\left(N(\mathfrak{a})^{1/n} t \lambda_1(\Psi_{\beta_s} A_t) \right)^2 &= \| \Psi_{\beta_s} \phi(\alpha) \|^2 \\
			&= \sum_{j=1}^{r_1+r_2} \vert \beta_ {s,j} \sigma_j(\alpha)  \vert^2 \\
			&\geq \frac{1}{2} \sum_{j=1}^{r_1+r_2} n_i \vert \beta_ {s,j} \sigma_j(\alpha)  \vert^2 \\
			& \geq \frac{n}{2}  \left( \prod \vert \beta_ {s,j} \sigma_j(\alpha) \vert^{2n_i} \right)^{1/n} \\
			&= \frac{n}{2}  \vert N_{K/\Q}(\alpha)\vert^{2/n}.
		\end{alignat*}
		Since $\alpha \in \mathfrak{a}$, we have $\vert N_{K/\Q}(\alpha)\vert / N(\mathfrak{a}) = N(\alpha \OO_K  \mathfrak{a}^{-1})$, and $\mathfrak{b}=\alpha \OO_K \mathfrak{a}^{-1}$ is an integral ideal in $\mathcal{O}_K$. Since $\mathfrak{a} \in \mathcal{C}^{-1}$, we have $\mathfrak{b}\in \mathcal{C}$. The inequality above then boils down to
		$\lambda_1(\Psi_{\beta_s} A_t) \geq \frac{\sqrt{n}}{\sqrt{2}}  N(\mathfrak{b})^{1/n}  t^{-1}$.
	\end{proof}

	\begin{lemma} \label{L:MinSumBound}
		There exist distinct integral ideals $\mathfrak{b}_1, \dots, \mathfrak{b}_k \in \mathcal{C}$  with $k\leq \vert S \vert$, such that 
		$$
		\sum_{s\in \mathcal{S}}\lambda_1(\Psi_{\beta_s} A_t)^{1-n}\leq 
		 \exp \left( \frac{(n-1) \sqrt{r}}{2c}\right) \frac{\vert S \vert}{R_K} \left(\frac{n}{n-1}\right)^r
		 \sum_{j=1}^k N(\mathfrak{b}_j)^{\frac{1-n}{n}}  t^{n-1}.
		$$
	\end{lemma}

    \begin{remark}
        Here we essentially adapt an argument that can be found in \cite[p.~13]{Korneel}. We show that even with our choice of fundamental domain of $\Gamma$ and our choice of Schmidt partition, we can use similar arguments to bound the sum $\sum_{s\in \mathcal{S}}\lambda_1(\Psi_{\beta_s} A_t)^{1-n}$ using an integral computed in \cite[Lemma~10]{Korneel}.
    \end{remark}
    
    \begin{proof}[Proof of Lemma \ref{L:MinSumBound}]
        The coordinates of a lattice point in $\Psi_{\beta_s} A_t$ have the form $\beta_{s,i} (N(\mathfrak{a})^{1/n}t)^{-1}\sigma_i(\alpha)$ for some element $\alpha\in \mathfrak{a}$. So we can bound
        $$
        \lambda_1(\Psi_{\beta_s} A_t)  (N(\mathfrak{a})^{1/n}t) \geq \min_{\alpha\in \mathfrak{a}\setminus\{0\}} \max_i (\vert \beta_{s,i} \vert  \vert \sigma_i(\alpha) \vert)=:\mu(\mathfrak{a},s).
        $$
        Hence, instead of bounding $\sum_{s\in \mathcal{S}} \lambda_1(\Psi_{\beta_s} A_t)^{1-n}$, we can now bound $\sum_{s\in \mathcal{S}} \mu(\mathfrak{a},s)^{1-n}$.

        For $s\in \mathcal{S}$, we can find $\alpha_s \in \mathfrak{a}$, such that $\mu(\mathfrak{a},s) = \max_i (\vert \beta_{s,i} \vert  \vert \sigma_i(\alpha_s) \vert)$. We fix a decomposition  $\mathcal{O}_K^\times = W \oplus U$, where $W$ is the subgroup of roots of unity and $U\cong \mathbb{Z}^r$. By possibly replacing some of the $\alpha_i$ by $\alpha_i\zeta$ with $\zeta\in W$, we can find $\alpha_1^\prime, \dots, \alpha_k^\prime \in \mathfrak{a}$ non-associated elements with $k\leq \vert S \vert$ together with $u_s \in U$ for $s\in S$, such that $\alpha_s= u_s  \alpha_{i_s}^\prime $. We now get
	\begin{alignat*}{2}
        \sum_{s\in \mathcal{S}} \mu(\mathfrak{a},s)^{1-n}
		& = \sum_{s\in \mathcal{S}} \max_i \left(\vert \beta_{s,i} \vert  \vert \sigma_i(\alpha_s) \vert \right)^{1-n} \\
		& \leq \sum_{j=1}^k \sum_{u\in U} \sum_{s\in \mathcal{S}} \max_i \left(\vert \beta_{s,i} \vert  \vert \sigma_i(u  \alpha_j^\prime) \vert \right)^{1-n} \\
		& = \sum_{j=1}^k \sum_{u\in U} \sum_{s\in \mathcal{S}} \max_i \left(\vert \beta_{s,i} \vert  \vert \sigma_i(u) \vert  \vert \sigma_i(\alpha_j^\prime) \vert \right)^{1-n}.
	\end{alignat*}
    Consider the group homomorphism $g_0=g(\cdot,0):\lsp \to\mkw^+$, which satisfies $\ell\circ g_0=\operatorname{id}_\lsp$, and $g_0\circ\ell$ is the same as applying $\vert\cdot\vert$ in every coordinate. Since $\ell\circ\phi$ is a bijection between $U$ and $\Gamma$, we can rewrite the sums over $U$ and $\mathcal{S}$ above as
    $$
    \sum_{u\in U} \sum_{s\in \mathcal{S}} \max_i \left(\vert \beta_{s,i} \vert  \vert \sigma_i(u) \vert  \vert \sigma_i(\alpha_j^\prime) \vert \right)^{1-n} = \sum_{l\in \Gamma} \sum_{s\in \mathcal{S}} \max_i \left(g_0(-y_s)_i g_0(l)_i  \vert \sigma_i(\alpha_j^\prime) \vert \right)^{1-n},
    $$
    where $\ell(\beta_s)= - y_s =  - \sum_{i=1}^{r} \frac{s_i}{\delta_i} e_i$. We want to approximate this sum by an integral. To do this, we extend the collection of vectors $y_s\in \lsp$, indexed by $s\in \mathcal{S}\subseteq\Z^r$, to a linear map
    $$
    \y: \R^r \to \lsp; \; \y(w)= \sum_{i=1}^{r} \frac{w_i}{\delta_i} e_i.
    $$
    Notice that $D_{\delta}=\{\y(w)\mid w\in \left[ -1/2,1/2 \right)^r\}$, so, for $w\in \left[ -1/2,1/2 \right)^r$, we have 
    $$
    g_0(\y(w))_i = \exp\left(\frac{\y(w)_i}{n_i}\right) \geq \exp\left(\y(w)_i\right) \geq \exp\left(- \frac{\sqrt{r}}{2c}\right).
    $$
    Therefore, we can bound
    \begin{alignat*}{2}
    &\max_i \left(g_0(-y_s)_i g_0(l)_i  \vert \sigma_i(\alpha_j^\prime) \vert \right)^{1-n} \\
    \leq &\exp\left( \frac{\sqrt{r}(n-1)}{2c}\right) \int_{w\in \left[ -1/2,1/2 \right)^r} \max_i \left( g_0(\y(w-s))_i g_0(l)_i  \vert \sigma_i(\alpha_j^\prime) \vert \right)^{1-n} \mathrm{d}w.
    \end{alignat*}
    Now, since $\{\y(w-s)\mid w\in \left[ -1/2,1/2 \right)^r,\ s\in \mathcal{S}\}$ is a fundamental domain of $\Gamma$, summing over $\mathcal{S}$ and $\Gamma$ will simply extend the domain of integration to all of $\R^r$, i.e.,
    $$
    \sum_{l\in \Gamma} \sum_{s\in \mathcal{S}}  \int_{w\in \left[ -1/2,1/2 \right)^r} \max_i \left( g_0(\y(w-s))_i g_0(l)_i  \vert \sigma_i(\alpha_j^\prime) \vert \right)^{1-n} \mathrm{d}w =
    \int_{w\in \R^r} \max_i \left( g_0(\y(w))_i  \vert \sigma_i(\alpha_j^\prime) \vert \right)^{1-n} \mathrm{d}w.
    $$
    This integral is essentially the same as \cite[Lemma~10]{Korneel}, i.e.,
    $$
\int_{w\in \R^r} \max_i \left( g_0(\y(w))_i  \vert \sigma_i(\alpha_j^\prime) \vert \right)^{1-n} \mathrm{d}w= \vert N_{K/\Q}(\alpha_j^\prime) \vert^{\frac{1-n}{n}} \frac{\vert S \vert}{R_K} \left(\frac{n}{n-1}\right)^r.
    $$
    Hence, we obtain
    \begin{alignat*}{2}
        \sum_{s\in \mathcal{S}}\lambda_1(\Psi_{\beta_s} A_t)^{1-n}&\leq N(\mathfrak{a})^{\frac{n-1}{n}}t^{n-1}  \sum_{s\in \mathcal{S}} \mu(\mathfrak{a},s)^{1-n}\\
        &\leq  N(\mathfrak{a})^{\frac{n-1}{n}}t^{n-1}\sum_{j=1}^k \sum_{u\in U} \sum_{s\in \mathcal{S}} \max_i \left(\vert \beta_{s,i} \vert  \vert \sigma_i(u) \vert  \vert \sigma_i(\alpha_j^\prime) \vert \right)^{1-n} \\
        &\leq N(\mathfrak{a})^{\frac{n-1}{n}}t^{n-1}\sum_{j=1}^k \exp\left( \frac{\sqrt{r}(n-1)}{2c}\right) \int_{w\in \R^r} \max_i \left( g_0(\y(w))_i  \vert \sigma_i(\alpha_j^\prime) \vert \right)^{1-n} \mathrm{d}w \\
        &\leq N(\mathfrak{a})^{\frac{n-1}{n}}t^{n-1}  \exp\left( \frac{\sqrt{r}(n-1)}{2c}\right) \sum_{j=1}^k \vert N(\alpha_j^\prime) \vert^{\frac{1-n}{n}} \frac{\vert S \vert}{R_K} \left(\frac{n}{n-1}\right)^r.
    \end{alignat*}
    
    Just as in the last lemma, we have $\vert N_{K/\Q}(\alpha_j^\prime)\vert / N(\mathfrak{a}) = N(\mathfrak{b}_j)$, where $\mathfrak{b}_j=\alpha_j^\prime \OO_K \mathfrak{a}^{-1}$ are distinct integral ideals in the class $\mathcal{C}$.
    \end{proof}

	\begin{lemma}[{\cite[Lemma~12]{Korneel}}]\label{L:SumIdeals} Let $\mathfrak{b}_1,\mathfrak{b}_2,\dots,\mathfrak{b}_k$ be distinct integral ideals in $\mathcal{O}_K$. Then
		$$
		\sum_{j=1}^{k} N(\mathfrak{b}_j)^\frac{1-n}{n} \leq 6n k^{1/n} \log^+(k)^{\frac{(n-1)^2}{n}}.
		$$
        Here, $\log^+(x) = \max(1,\log(x))$.
	\end{lemma}
	
	\section{Counting Theorem}\label{S:main}
	We can now prove our main theorem, which we restate here for convenience.
	\begin{mainth}
		For $t\geq \left(10 n^2 \sqrt{\Delta_K} \right)^n$, we have
        \begin{alignat}{3}
            \vert \sigma_K(t,\mathcal{C}) - t \, h_K^{-1}\kappa_K  \vert &\leq  t^{\frac{n-1}{n}}\,E_1(n)\, \omega_K^{-1} R_K , \tag{A1} \label{E:a1} \\
            \vert \sigma_K(t,\mathcal{C}) - t \, h_K^{-1}\kappa_K  \vert &\leq  t^{\frac{n-1}{n}}\,  E_2(n) \,  \omega_K^{-1}   R_K^{1/n}  \log^+\left(\left( 15n\, 2^{\frac{r}{2}} \right)^n R_K \right)^{\frac{(n-1)^2}{n}}  ,  \tag{A2} \label{E:a2} \\
			\vert \sigma_K(t) - t \, \kappa_K  \vert 
            &\leq t^{\frac{n-1}{n}}\, 6nE_1(n)\, \omega_K^{-1}   R_K \, h_K^{1/n}   \log^+(h_K)^{\frac{(n-1)^2}{n}}, \tag{B1} \label{E:b1} \\
			\vert \sigma_K(t) - t \, \kappa_K  \vert 
            &\leq t^{\frac{n-1}{n}}\, E_2(n)  \, \omega_K^{-1}   (R_K h_K)^{1/n}  \log^+\left(\left( 15n\, 2^{\frac{r}{2}} \right)^n R_K h_K \right)^{\frac{(n-1)^2}{n}} . \tag{B2} \label{E:b2}
        \end{alignat}
	Here, $\log^+(x) = \max(1,\log(x))$, and we can take 
    $$E_1(n)=\frac{45e}{2} n^{7/2} 2^{4.1  n} \quad \text{ and } \quad E_2(n)=2025e^2  n^{11/2} 2^{4n} \left(n-1  \right)^{\frac{n-1}{2}}.$$
	\end{mainth}

	\begin{proof}
		By Proposition \ref{P:SigLat}, we have
		$$
		\sigma_K(t^n,\mathcal{C}) = \frac{1}{\omega_K} \sum_{s \in \mathcal{S}} \vert (\Psi_{\beta_s} A_t) \cap \F( D_{\delta},\left(0,1\right],\nu) \vert,
		$$
        and, from Corollary \ref{C:LipConstF}, we know that 
		$\partial  \F( D_{\delta},\left(0,1\right],\nu)$ is of Lipschitz class $\operatorname{Lip}\left(n-1,M,L \right)$
		with
		$M\leq6n \sqrt{r_1+r_2} c  \left(2\pi c\right)^{r_2} 2^{r_1} $
		and
		$L=\exp \left( \frac{\sqrt{r}}{2c} \right) c^{-1}$. We use Proposition \ref{P:LatLipCounting} to estimate
		$$
		\bigg\vert \vert (\Psi_{\beta_s} A_t) \cap \F( D_{\delta},\left(0,1\right],\nu) \vert - \frac{\vol(\F( D_{\delta},\left(0,1\right],\nu))}{\vol(\Psi_{\beta_s} A_t)} \bigg\vert \leq M \  \left(\sqrt{n-1} L \right)^{n-1} 2^n \vol \left(B (1) \right)  \frac{r(\Psi_{\beta_s} A_t)}{\vol(\Psi_{\beta_s} A_t)} ,
		$$
        which is conditional on $r(\Psi_{\beta_s} A_t) \leq \sqrt{n-1} L$.
        This condition gives us a lower bound on $t$, which depends on the choice of the parameter $c$. We will first calculate our error term, optimising the choice of $c$, and afterwards calculate the lower bound on $t$ that this choice gives.
        
		Combining Proposition \ref{P:volF} and $\vol(\Psi_{\beta_s} A_t)=\vol(A_t)=\frac{\sqrt{\Delta_K}}{2^{r_2}  t^n}$, we can see that the main term in our estimation of $\sigma_K(t^n,\mathcal{C})$ is $t^n \, h_K^{-1}\kappa_K$. Using  Lemma \ref{L:RadMin} to bound $r(\Psi_{\beta_s} A_t)$ in terms of $\lambda_1(\Psi_{\beta_s} A_t)$ we obtain
		$$
		\vert \sigma_K(t^n,\mathcal{C}) - t^n \, h_K^{-1}\kappa_K  \vert \leq \omega_K^{-1}  M  \left(\sqrt{n-1} L \right)^{n-1}  2^{2n-1} \sqrt{n} \sum_{s \in \mathcal{S}} \lambda_1(\Psi_{\beta_s} A_t)^{1-n}.
		$$
		We can use two different approaches to bound the sum
		$\sum_{s \in \mathcal{S}} \lambda_1(\Psi_{\beta_s} A_t)^{1-n}$.
        
		The first is using Lemma \ref{L:MinBound}, which yields
		$$
		\sum_{s \in \mathcal{S}} \lambda_1(\Psi_{\beta_s} A_t)^{1-n} \leq  \left(\frac{2}{n}\right)^{\frac{n-1}{2}}  \vert \mathcal{S}  \vert  N(\mathfrak{b}_\mathcal{C})^{\frac{1-n}{n}}  t^{n-1},
		$$
		where $\mathfrak{b}_\mathcal{C} \subseteq \mathcal{O}_K$ is an integral ideal in the class $\mathcal{C}$.

        Now we substitute our values for $M$ and $L$, use the bound for $\vert \mathcal{S}\vert$ from Lemma \ref{L:SBound}, and optimise the choice of our parameter $c$. We outsource the simplification of these terms to Lemma \ref{L:ugly}. Inequality \eqref{E:c1} from the lemma gives
        $$
        \vert \sigma_K(t,\mathcal{C}) - t \, h_K^{-1}\kappa_K  \vert \leq  t^{\frac{n-1}{n}}\,E_1(n) \omega_K^{-1} R_K    N (\mathfrak{b}_\mathcal{C})^{\frac{1-n}{n}} .
        $$
		
        Bounding $N(\mathfrak{b}_\mathcal{C})^{\frac{1-n}{n}}\leq 1$, we obtain our first bound \eqref{E:a1}.
        When we are summing over the different ideal classes to estimate $\sigma_K(t^n)$, we use Lemma \ref{L:SumIdeals} to bound the sum over $h_K$ distinct ideals, i.e.,
		$$
		\sum_{\mathcal{C} \in \operatorname{Cl}(\OO_K)}  N(\mathfrak{b}_\mathcal{C})^{\frac{1-n}{n}} \leq 6n h_K^{1/n} \log^+(h_K)^{\frac{(n-1)^2}{n}}.
		$$
		This then yields result \eqref{E:b1}.
		
		The second option for bounding $\sum_{s \in \mathcal{S}} \lambda_1(\Psi_{\beta_s} A_t)^{1-n}$ is to use Lemma \ref{L:MinSumBound}, which guarantees that we can find distinct integral ideals $\mathfrak{b}_1,\dots,\mathfrak{b}_k \in \mathcal{C}$ with $k\leq \vert \mathcal{S} \vert$, such that
		$$
		\sum_{s\in \mathcal{S}}\lambda_1(\Psi_{\beta_s} A_t)^{1-n}\leq 
		 \exp \left( \frac{(n-1) \sqrt{r}}{2c}\right) \frac{\vert S \vert}{R_K} \left(\frac{n}{n-1}\right)^r
		 \sum_{j=1}^k N(\mathfrak{b}_j)^{\frac{1-n}{n}}  t^{n-1}.
		$$
        We estimate the sum over $k$ distinct ideals using Lemma \ref{L:SumIdeals}, i.e.,
        $$
        \sum_{j=1}^{k} N(\mathfrak{b}_j)^\frac{1-n}{n} \leq 6n k^{1/n} \log^+(k)^{\frac{(n-1)^2}{n}}\leq 6n \vert \mathcal{S}\vert^{1/n} \log^+(\vert \mathcal{S}\vert)^{\frac{(n-1)^2}{n}}.
        $$
        Using this together with our values and bounds for $M,L,c,$ and $\vert \mathcal{S}\vert$ gives rise to result \eqref{E:a2}. Again we outsource the uninspiring calculations to Lemma \ref{L:ugly} (see inequality \eqref{E:c2}).
        
		When we use this result to estimate $\sigma_K(t^n)$, we will again apply Lemma \ref{L:SumIdeals}, but this time we sum over $k  h_K$ distinct ideals.
		This yields the last estimation \eqref{E:b2}.

        Finally, we check, using Lemmas \ref{L:RadMin} and \ref{L:MinBound}, that with our choice of the parameter $c$, we have 
        $r(\Psi_{\beta_s} A_t) \leq \sqrt{n-1} L$ as soon as $t\geq 10 n^2 \sqrt{\Delta_K}$. This calculation is the last part of Lemma \ref{L:ugly}.
	\end{proof}

    \begin{lemma}\label{L:ugly}
    With
    $M \leq  6n\sqrt{r_1+r_2} c \left(2\pi c\right)^{r_2} 2^{r_1} $
		,
		$L=\exp \left( \frac{\sqrt{r}}{2c} \right) c^{-1}$ as in Corollory \ref{C:LipConstF}, and
        $\vert S\vert \leq \exp\left(\frac{14 n 2^{\frac{r}{2}}}{c} \right)  c^r \sqrt{r_1+r_2} R_K$ from Lemma \ref{L:SBound},
        we can choose the constant $c\geq 1$, such that $2\pi c\in \mathbb{N}$,
    \begin{equation}\tag{C1} \label{E:c1}
    \omega_K^{-1}  M  \left(\sqrt{n-1} L \right)^{n-1}  2^{2n-1} \sqrt{n}   \left(\frac{2}{n}\right)^{\frac{n-1}{2}}  \vert \mathcal{S}  \vert  N(\mathfrak{b}_\mathcal{C})^{\frac{1-n}{n}}  t^{n-1} 
    \leq E_1(n)   \frac{R_K}{\omega_K}   N(\mathfrak{b}_\mathcal{C})^{\frac{1-n}{n}}  t^{n-1}
    \end{equation}
    with $E_1(n)=\frac{45e}{2} n^{7/2} 2^{4.1  n}$, and
    \begin{alignat}{2}\tag{C2} \label{E:c2}
        &\omega_K^{-1}  M  \left(\sqrt{n-1} L \right)^{n-1}  2^{2n-1} \sqrt{n}  \exp \left( \frac{(n-1) \sqrt{r}}{2c}\right) \frac{\vert S \vert}{R_K} \left(\frac{n}{n-1}\right)^r
		 6n \vert \mathcal{S}\vert^{1/n} \log^+(\vert \mathcal{S}\vert)^{\frac{(n-1)^2}{n}}  t^{n-1} \\
         &\leq  E_2(n)
		  \omega_K^{-1}R_K^{1/n} \log^+\left( \left( 15n 2^{\frac{r}{2}} \right)^n R_K \right)^{\frac{(n-1)^2}{n}}  t^{n-1} \notag
    \end{alignat}
    with $E_2(n)=2025e^2  n^{11/2} 2^{4n} \left(n-1  \right)^{\frac{n-1}{2}}$.

    With this choice of $c$, we can ensure that $r(\Psi_{\beta_s} A_t) \leq \sqrt{n-1} L$ as soon as $t\geq 10 n^2 \sqrt{\Delta_K}$.
    \end{lemma}
    \begin{proof}
    We start by simplifying some terms. Let $Q=M  \left(\sqrt{n-1} L \right)^{n-1}  \vert \mathcal{S}  \vert  2^{2n-1} \sqrt{n}$.
    Then
    \begin{alignat*}{2}
			Q &\leq   6n \sqrt{r_1+r_2} c \left(2\pi c\right)^{r_2} 2^{r_1}  \left(\sqrt{n-1} \exp \left( \frac{\sqrt{r}}{2c} \right) c^{-1} \right)^{n-1}  \exp\left(\frac{14 n 2^{\frac{r}{2}}}{c} \right)  c^r \sqrt{r_1+r_2} R_K  2^{2n-1} \sqrt{n} \\
            &=3n^{3/2} (r_1+r_2) \left(2\pi\right)^{r_2} 2^{2n+r_1}  \left(n-1  \right)^{\frac{n-1}{2}}  c \exp \left(\frac{\sqrt{r}(n-1) +28 n 2^{\frac{r}{2}}}{2c} \right)  R_K.
		\end{alignat*}
        Now we can finally see what a good choice for our parameter $c$ is. We want to minimise a term of the form $c\exp(ac^{-1})$. This is minimised for $c=a$, where the value attained is $e a$.

        So we want to choose $c$ on the order of $\frac{\sqrt{r}(n-1) +28 n 2^{\frac{r}{2}}}{2}$. Recall that, in Subsection \ref{Ss:LipComp}, we imposed the restriction $2\pi c\in \mathbb{N}$. But even with this additional constraint, we can certainly choose $c$, such that $\frac{\sqrt{r}(n-1) +28 n 2^{\frac{r}{2}}}{2} \leq c \leq 15n 2^{\frac{r}{2}}$, and obtain
        $$
    c \exp \left(\frac{2\sqrt{r}(n-1) +28 n 2^{\frac{r}{2}}}{2c} \right) \leq 15 e n  2^{r/2}.
        $$
        With this, our bound for $Q$ becomes
        \begin{alignat*}{2}
        Q &\leq   45e n^{5/2} (r_1+r_2) \left(2\pi\right)^{r_2} 2^{r_1 +2n +r/2}  \left(n-1  \right)^{\frac{n-1}{2}}    R_K  \exp \left( \frac{-\sqrt{r}(n-1)}{2c} \right)\\
        &\leq 45e n^{7/2} \left(2\pi\right)^{r_2} 2^{r_1 +2n +r/2}  \left(n-1  \right)^{\frac{n-1}{2}}    R_K  \exp \left( \frac{-\sqrt{r}(n-1)}{2c} \right) .
        \end{alignat*}
    This gives us our bound \eqref{E:c1}
    \begin{alignat*}{2}
    &\omega_K^{-1}  M  \left(\sqrt{n-1} L \right)^{n-1}  2^{2n-1} \sqrt{n}   \left(\frac{2}{n}\right)^{\frac{n-1}{2}}  \vert \mathcal{S}  \vert  N(\mathfrak{b}_\mathcal{C})^{\frac{1-n}{n}}  t^{n-1} \\
    &=
    Q \omega_K^{-1}  \left(\frac{2}{n}\right)^{\frac{n-1}{2}}  N(\mathfrak{b}_\mathcal{C})^{\frac{1-n}{n}}  t^{n-1}\\
    &\leq 45e n^{7/2} \left(2\pi\right)^{r_2} 2^{r_1 +2n +r/2}  \left(n-1  \right)^{\frac{n-1}{2}}    \frac{R_K}{\omega_K}  \left(\frac{2}{n}\right)^{\frac{n-1}{2}}  N(\mathfrak{b}_\mathcal{C})^{\frac{1-n}{n}}  t^{n-1}\\
    &= \frac{45e}{2} n^{7/2} 2^{4r_1 + \left( \frac{\log 2\pi}{\log 2} +\frac{11}{2}\right) r_2}    \frac{R_K}{\omega_K}  \left(\frac{n-1}{n}\right)^{\frac{n-1}{2}}  N(\mathfrak{b}_\mathcal{C})^{\frac{1-n}{n}}  t^{n-1}\\
    &\leq \frac{45e}{2} n^{7/2} 2^{4.1 n}    \frac{R_K}{\omega_K}   N(\mathfrak{b}_\mathcal{C})^{\frac{1-n}{n}}  t^{n-1}.
    \end{alignat*}
    
    Now we turn our attention to \eqref{E:c2}.
    With our choice of $c$, we obtain an upper bound for the size of $\mathcal{S}$, i.e.,
    $$
    \vert \mathcal{S}\vert \leq  \exp\left(\frac{14 n 2^{\frac{r}{2}}}{c} \right)  c^r \sqrt{r_1+r_2} R_K \leq e  \left( 15n 2^{\frac{r}{2}} \right)^r \sqrt{r_1+r_2} R_K \leq \left( 15n 2^{\frac{r}{2}} \right)^n R_K.
    $$
    Using this, we obtain inequality \eqref{E:c2}
    \begin{alignat*}{2}
        &\omega_K^{-1}  M  \left(\sqrt{n-1} L \right)^{n-1}  2^{2n-1} \sqrt{n}  \exp \left( \frac{(n-1) \sqrt{r}}{2c}\right) \frac{\vert S \vert}{R_K} \left(\frac{n}{n-1}\right)^r
		 6n \vert \mathcal{S}\vert^{1/n} \log^+(\vert \mathcal{S}\vert)^{\frac{(n-1)^2}{n}}  t^{n-1} \\
         &= \omega_K^{-1} R_K^{-1}  Q  \exp \left( \frac{(n-1) \sqrt{r}}{2c} \right) \left(\frac{n}{n-1}\right)^r
		 6n \vert \mathcal{S}\vert^{1/n} \log^+(\vert \mathcal{S}\vert)^{\frac{(n-1)^2}{n}}  t^{n-1} \\
         &\leq \omega_K^{-1} 45e n^{7/2} \left(2\pi\right)^{r_2} 2^{r_1 +2n +r/2}  \left(n-1  \right)^{\frac{n-1}{2}}  
          \left(\frac{n}{n-1}\right)^r
		 6n  15n 2^{\frac{r}{2}} R_K^{1/n} \log^+\left( \left( 15n 2^{\frac{r}{2}} \right)^n R_K \right)^{\frac{(n-1)^2}{n}}  t^{n-1} \\
         &=2025e  n^{11/2}
         2^{4r_1+\left(\frac{\log 2\pi}{\log 2} + 5 \right)r_2}
         \left(n-1  \right)^{\frac{n-1}{2}}  
          \left(\frac{n}{n-1}\right)^r
		  \omega_K^{-1}R_K^{1/n} \log^+\left( \left( 15n 2^{\frac{r}{2}} \right)^n R_K \right)^{\frac{(n-1)^2}{n}}  t^{n-1} \\
         &\leq 2025e^2  n^{11/2}
         2^{4n}
         \left(n-1  \right)^{\frac{n-1}{2}} 
		  \omega_K^{-1}R_K^{1/n} \log^+\left( \left( 15n 2^{\frac{r}{2}} \right)^n R_K \right)^{\frac{(n-1)^2}{n}}  t^{n-1}.
    \end{alignat*}

    Finally, we look at the inequality  $r(\Psi_{\beta_s} A_t) \leq \sqrt{n-1} L$. We use Lemmas \ref{L:RadMin}, and \ref{L:MinBound} to bound
    
    $$r(\Psi_{\beta_s} A_t) \leq \frac{ 2^{n-1} \sqrt{n} }{ \vol(B(1))} \frac{\vol(\Psi_{\beta_s} A_t)}{\lambda_1(\Psi_{\beta_s} A_t)^{n-1}} \leq \frac{ 2^{n-1} \sqrt{n} }{ \vol(B(1))} \frac{\sqrt{\Delta_K}}{2^{r_2}  t^n} \left(\frac{2}{n}\right)^{\frac{n-1}{2}} t^{n-1}.$$
    Using Stirling's approximation of the Gamma-function, we can bound $\vol(B(1))^{-1}$ and simplify some terms to obtain
    $$
    r(\Psi_{\beta_s} A_t) \leq (e\pi)^{\frac{-n}{2}}   \left(\frac{n}{2}\right)^{\frac{n+1}{2}} \sqrt{2\pi} \,e^{\frac{1}{6n}} 2^{n-1} \sqrt{n}  \frac{\sqrt{\Delta_K}}{2^{r_2}  t^n} \left(\frac{2}{n}\right)^{\frac{n-1}{2}} t^{n-1} 
    =
    \frac{\sqrt{2\pi} \,e^{\frac{1}{6n}}}{2} n^{3/2}(e\pi)^{\frac{-n}{2}}    2^{r}   \sqrt{\Delta_K}\, t^{-1}.
    $$
    On the other hand, with the bound $c \leq 15n 2^{\frac{r}{2}}$, we have 
    $$
    \sqrt{n-1}L=\sqrt{n-1}\exp \left( \frac{\sqrt{r}}{2c} \right) c^{-1} \geq  \sqrt{n-1} \exp \left( \frac{\sqrt{r}}{30n 2^{\frac{r}{2}}} \right) \left (15n 2^{\frac{r}{2}} \right)^{-1}.
    $$

    Thus, we can ensure that $r(\Psi_{\beta_s} A_t) \leq \sqrt{n-1} L$ as soon as $t$ is bigger than
    
    $$
    \frac{\sqrt{2\pi} \,e^{\frac{1}{6n}}}{2} n^{3/2}(e\pi)^{\frac{-n}{2}}  2^{r}  \sqrt{\Delta_K} \sqrt{n-1}^{-1} \exp \left( \frac{-\sqrt{r}}{30n 2^{\frac{r}{2}}} \right) \left (15n 2^{\frac{r}{2}} \right),
    $$
    which we can bound above by the simpler term $10 n^2 \sqrt{\Delta_K}$.
    \end{proof}

	\section{Discussion}\label{S:discuss}

    \begin{remark}
    We shortly discuss how our two estimates \eqref{E:b1} and \eqref{E:b2} from Theorem \ref{T:main} compare to each other. In particular, we want to discuss under which circumstances one or the other estimate can be expected to perform better.

    Given a single number field $K$, where all invariants $n,\omega_K,h_K$, and $R_K$ are computable, one can, of course, just compare directly. However we want to remark that there are some bounds on the residue $\kappa_K$, such as Louboutin's bound \cite{louboutin}
        $$
	    \kappa_K \leq \left(\frac{e\log(\Delta_K)}{2(n-1)} \right)^{n-1},
	    $$
    with which we can roughly estimate $h_KR_K\sim \sqrt{\Delta_K}$. We can use this to estimate the main contributors of the error term in terms of the degree $n$ and the discriminant $\Delta_K$. We can expect that the constant of the error term in \eqref{E:b1} is very roughly on the order of
    $2^{4.1 n}\sqrt{\Delta_K}$,
    while the constant of the error term in \eqref{E:b2} is rather on the order of $n^{2.5n} \Delta_K^{\frac{1}{2n}}$. 

    So, heuristically, we can expect \eqref{E:b1} to perform better for ``small'' discriminants $\Delta_K \ll n^n$, while \eqref{E:b2} is better suited for larger discriminants $\Delta_K\gg n^{n^2}$.
    \end{remark}

    The comparison to Debaene's bound (Theorem \ref{T:Debaene}) is rather simple. We managed to reduce the dependence of the error term on the degree of the extension, while retaining the same dependence on the regulator $R_K$ and the class number $h_K$. However, our bound only holds for $t\geq \left(10 n^2 \sqrt{\Delta_K} \right)^n$, while Debaene's result holds for all $t\geq 1$.

    Lastly, we want to discuss how our result compares to the following explicit version of Landau's estimate $\vert \sigma_K(t) -  t \kappa_K \vert \ll t^{\frac{n-1}{n+1}}$, which is due to Lee.

    \begin{theorem}[{\cite[Theorem~1.2]{lee_number_2023}}]\label{T:Lee}
        For $t>0$ and $n\geq 2$, we have
        $$
        \vert\sigma_K(t)-t \,\kappa_K \vert  \leq   t^{\frac{n-1}{n+1}} \Lambda_K(n) \Delta_K^{\frac{1}{n+1}} \log(\Delta_K)^{n-1}.
        $$
    \end{theorem}
    For the explicit description of the function $\Lambda_K(n)$, we refer to the article.

    It is clear that the error term grows slower in $t$, so for very large values of $t$ the estimate of Lee will certainly outperform our result. However, for values of $t$ close to our lower bound, $\left(10 n^2 \sqrt{\Delta_K} \right)^n$, we still think, it is interesting to see how the two very different approaches of giving explicit estimates for $\sigma_K(t)$ compare. We do this by giving an example for a cyclotomic extension.
    
    \begin{example}\label{Ex:zeta}
        For $K=\Q(\zeta_{11})$ of degree $n=10$, we can find all relevant constants on the LMFDB \cite[\href{https://www.lmfdb.org/NumberField/10.0.2357947691.1}{Number field 10.0.2357947691.1}]{lmfdb}.

        In \cite{lee_number_2023}, we can luckily find $\Lambda_K(10)=9.65555 \cdot 10^{26}$ in Table 1. For $t_0=\left(10 n^2 \sqrt{\Delta_K} \right)^n \approx 7.29 \cdot 10^{76}$, Theorem \ref{T:Lee} gives
        $$
        \vert \sigma_K(t_0) -\kappa_K\,t_0 \vert \leq  5.39\cdot 10^{102}.
        $$
        Since $h_K=1$, i.e., we only have one ideal class, and $R_K=26.1711060094$ is rather small, the best choice from among our bounds is \eqref{E:a1}.
        With this, we obtain
        $$
        \vert \sigma_K(t_0) -\kappa_K\,t_0 \vert \leq 7.6 \cdot 10^{86}.
        $$
    \end{example}
	
	\bibliographystyle{alpha}
	\bibliography{biblio.bib}
	
\end{document}